\documentclass[12pt,reqno]{amsart}
\calclayout

\usepackage{amsthm,amsmath,amsfonts,amssymb, mathtools,dsfont}
\usepackage{graphicx}
\usepackage[mathcal]{euscript}
\usepackage{lmodern,mathrsfs}
\usepackage{tikz,ifthen}
\usetikzlibrary{calc,math}
\usepackage{amsrefs}
\usepackage[colorlinks,citecolor=blue,urlcolor=blue, linkcolor=blue]{hyperref}

\theoremstyle{plain}
\newtheorem{theorem}{Theorem}[section]

\newtheorem{lemma}[theorem]{Lemma}

\newtheorem{conjecture}[theorem]{Conjecture}
\newtheorem{corollary}[theorem]{Corollary}

\theoremstyle{remark}
\newtheorem{definition}[theorem]{Definition}
\newtheorem{example}[theorem]{Example}
\newtheorem{remark}[theorem]{Remark}

\newcommand{\Rd}{\mathbb{R}^d}

\DeclareMathOperator{\inl}{\mathtt{inl}}
\DeclareMathOperator{\scf}{\mathtt{scf}}
\DeclareMathOperator{\ctr}{\mathtt{c}}
\DeclareMathOperator{\card}{\mathtt{card}}

\newcommand{\seceq}{\setcounter{equation}{0}} 

\newcommand{\rev}[1]{{#1}}

\newcommand{\ts}{\textsuperscript}
\newcommand{\pr}{\mathbb P}
\newcommand{\zpl}{\mathbb{N}}
\newcommand{\rpl}{\mathbb{R}^+}
\newcommand{\ex}{\operatorname{\mathbb E}}

\newcommand{\der}{\mathrm d}
\newcommand{\mc}{\mathcal}

\newcommand{\bb}{\mathbb}
\newcommand{\ind}{\operatorname{\mathds{1}}}

\newcommand{\bs}{\mathbf}
\newcommand{\nn}{\nonumber}
\newcommand{\overbar}[1]{\mkern 1.5mu\overline{\mkern-1.5mu#1\mkern-1.5mu}\mkern 1.5mu}

\begin{document}
\title{Ising Disks: Topology Preserving Glauber Dynamics}

\author{Yuliy Baryshnikov}
\address{Department of Mathematics and 
ECE, University of Illinois Urbana-Champaign}
\email{ymb@uiuc.edu}
\thanks{This work was supported in part by ARO Grant W911NF-18-1-0327, MURI SLICE}

\author{Efe Onaran}
\address{Department of Mathematics and School of Engineering and Applied Science, University of Pennsylvania}
\email{eonaran@seas.upenn.edu}


\begin{abstract}
We introduce a dynamic model where the state space is the set of contractible cubical sets in the Euclidian space. \rev{The permissible state transitions, that is addition and removal of a cube to/from the set, are closest to \emph{Eden model} with topological constraints, and, we show, are \emph{locally decidable}}. We prove that in the planar special case the state space is connected. We then define a continuous time Markov chain with a \emph{fugacity} (tendency to grow) parameter. \rev{Using the correspondence between our model on the plane and the \emph{self-avoiding polygons},} we prove that the Markov chain is irreducible (due to state connectivity), and is also ergodic if the fugacity is smaller than a threshold.
\end{abstract}

\maketitle

\section{Introduction}
\seceq

A cubical set (complex) is a collection of cubes (i.e., products of unit intervals, of various dimensions) such that the intersection of any two cubes is again within the collection. In this paper, we will be dealing with complexes formed by the $d$-dimensional cubes of the integer partition of $\Rd$, cubes that are products of intervals that start and end at integers in $d$ copies of the real line. Equivalently, we will be dealing with the subsets of the \rev{simple cubic} lattice.

We will define a dynamic model on the set of cubical sets evolving through addition and removal of individual cubes. An important constraint on the cubical sets of our model is that {\em their topology will be fixed}: the homotopy type of the union of (closed) cubes of the complex will remain constant. In the following we will be initiating the process with a single cube, so that throughout the evolution, the cubical set will remain contractible. Of course, one can start with arbitrary topological homotopy type realizable by a cubical set in $\Rd$. 

We will be looking at the Markov chains valued in such cubical sets, defined by a Glauber-like local dynamics, generating, in essence, a topologically constrained Ising model. \rev{As in the Glauber dynamics for the Ising model, the state transition probabilities will be defined \emph{locally}, notwithstanding keeping the global topology fixed, through the detailed balance equations obtained from the stationary distribution, when it exists. Our treatment is closely in line with the recent interest in the problem of sampling of Ising model under hard global constraints (also called truncated Ising model, see \cite{chauhanlearning, carlsonfixedising}).}

It might be surprising that this constraint can be realized using local updates: for example, that would be impossible if one attempted to keep the configuration merely connected. However, as we show, the homotopy type can be preserved under the local moves, resulting in an ensemble of contractible cubical complexes. In fact, we will be working with a version where even the local relative homology of any boundary point of the cubical complex $\bs X$ modulo its small vicinity (in $\bs X$) is contractible.
 We call such cubical sets {\em clumps}.

In the case of planar clumps (see Figure~\ref{clumpfig}) their boundary forms a self-avoiding polygon, a close relative of the self-avoiding walks, a highly studied object in statistical physics \cite{guttbook, guttmann_self-avoiding_2012}. For $d=3$, clumps are a proper subset of \emph{polycubes}, cubes attached through their faces without regard to the overall topology \cite{polycube}. Our construction, in particular, allows one to generate higher-dimensional analogues to self-avoiding walks, self avoiding codimension-$1$ cellular structures, something that has not been explored thus far.

\begin{figure}[ht]
\centering
\includegraphics[width=.7\textwidth]{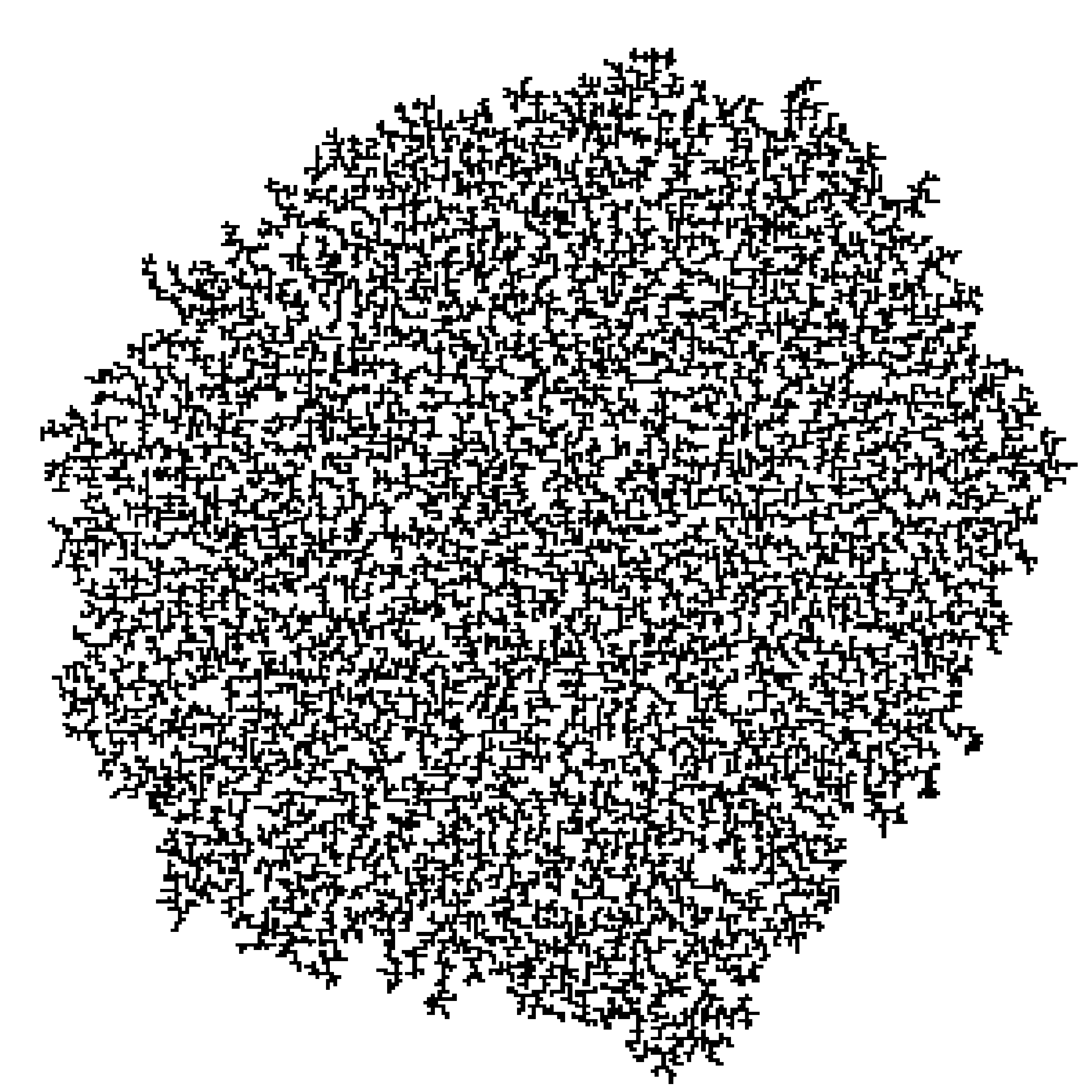}
    \caption{A planar clump, generated by the Markov process defined in Section~\ref{Isingsec} in the ``over-critical regime''. Also see Figure~\ref{snapshots}.}   
    \label{clumpfig}
\end{figure}

While our model is, we believe, new, there are several adjacent lines of research. 
The topology of general growth models has recently garnered attention. \rev{Classical stochastic growth models on $d$-dimensional lattice start with a cubical set (individual cube at the origin or a strip) and evolve by adding a randomly chosen cube at each step. The rules of choosing the cube to add are non-topological and vary by model. Examples include \emph{Eden model} \cite{eden}, \emph{dielectric breakdown model} \cite{dielectric}, \emph{diffusion-limited aggregation} \cite{wittendla}, \emph{ballistic deposition} \cite{meakin1986ballistic}, and their numerous variants \cite{lawleridla, atarstrip}. Specifically, in a recent paper \cite{damron_number_2023}, the authors consider the shape emerging by time $t$ in the Eden model (equivalently the ``least-cost ball'' in the first passage percolation model), and ask what is the number of holes in (equivalently, what is the $(d-1)$-th Betti number of) the resulting shape. More detailed results about the topology of that shape are obtained in \cite{manin_topology_2023,hua_local_2024}. Again, unlike ours, the emphasis in this line of research is to describe the local topology in the growth models, rather than to generate a topology-constrained one.}

\rev{Another line of literature that closely relates to our work here is Pach's animal problem \cite{tancer}. An animal is a cubical set in $d$-dimensional simple cubic lattice that is homeomorphic to the $d$-ball. The problem asks whether all animals can be reduced to a single cube by adding or removing cubes one-by-one and staying as an animal. For $d=2$, our main object of interest is an animal. Our Theorem~\ref{indentthm} gives an affirmative answer to a stronger version of
Pach's animal problem in the planar special case. Indeed, it shows that a planar animal can be reduced to \emph{any chosen cube} of its by only removals.  
For $d\geq 3$, Pach's problem is unsolved. It is not trivial even to see whether our model of clumps (and the definition of cube removal/addition) is equivalent to that of animals in $d=3$ case. Delineating this relationship has the prospect of shedding a fresh light to the Pach's problem and is left for future work. 
}

In this paper, we are focusing on the few basic patterns of the clumps: construction, phase transition in $d=2$, and some extremal event estimates. \rev{Specifically, we will first prove that the ``cube removal'' (or equivalently ``cube addition'') step that keeps the structure clump is locally decidable. Then we will prove that, in the special case of $d=2$, every clump has at least two cubes that can be ``removed'' which will let us establish the ergodicity of the Glauber-type dynamics we will define if the ratio of additions to removals is high enough. Lastly, given the ergodicity, we will present a ``rare-event'' type of result where we will show the convergence (in distribution) of the ``unusually large'' sightings of the dynamic planar clump to the Poisson point process.}

Our model is rich and invites a number of different questions. One of them is the question of the shapes of the clumps. One way to think about them is to view them as two counteracting drifts: to grow, driven by the entropic factor, i.e., the number of clump configurations growing exponentially, as $\mu^A, \mu>1$, in the surface area $A$, and the \emph{fugacity}, which controls the growth exponentially in the volume. As the volume cannot grow more slowly than the surface area of the clump, there is a phase transition: when the fugacity is below a threshold, $\kappa_*$, the clump remains bounded. Furthermore, in the over-critical regime, when the fugacity is not too high, the clumps grow in a porous fashion, looking like a coral (or {\em vesicle} \cite{guttmann_self-avoiding_2012}), see Fig. \ref{clumpfig}. It would be instructive to quantify this behavior.

While we focus here on the clumps evolving from a single cube, other starting configurations are possible. One version of interest is to start with a half space occupied, half vacant. A coral-like interface develops: what is its typical width? speed? These and many other questions will be addressed elsewhere.

\section{Cubical Set Definitions and Notation}
\seceq

Some of the standard definitions we include in the coming paragraphs, and will use throughout the paper, are from \cite{mischai}, although there are some important differences. 

An elementary interval is a set in the form of either $[l-\frac{1}{2},l +\frac{1}{2}]$ or $\{l+\frac{1}{2}\}$ for $l\in \bb Z$. It is called \emph{non-degenerate} if it is in the $[l-\frac{1}{2},l+ \frac{1}{2}]$ form, and \emph{degenerate} otherwise. For some fixed $d\geq 2$, an \emph{(elementary) cube}, $X$, is defined as a product of $d$ elementary intervals. We refer to the geometric realization of a given cube $X$ in the Euclidian space as $|X|\subset \Rd$. All cube realizations we consider will be in $\Rd$. 

We call an elementary cube \emph{$k$-dimensional} if exactly $k$ of the elementary intervals that compose it are non-degenerate. The set of all elementary cubes of dimension $j\geq 0$ is denoted by $\mc{K}^j$.
With some abuse of notation, for $X,Y\in\mc K^{j}$, we mean by $X\cap Y \in \mathcal{K}^{i}$ that $|X| \cap |Y|$ is the geometric realization of an $i$-dimensional elementary cube in $\Rd$. We will also say that $X$ and $Y$ \emph{share an $i$-dimensional face} in this case. 

We call the point in $\Rd$ obtained through taking the middle point of each elementary interval that composes a cube $X$, the \emph{center of $X$}, and denote it by $\ctr(X)$. For instance, the center of the cube $[-\frac{1}{2},\frac{1}{2}]\times\{\frac{3}{2}\}$ is the point 
$(0\;\; \frac{3}{2})\in\mathbb{R}^2$.

A \emph{cubical set}, $\bs X= \{X_1,\ldots ,X_m\}$ for $m\geq 2$, is any finite subset of $\mathcal{K}^d$. The cardinality of a cubical set will be denoted by $\card(\bs X)$, whereas $|\bs X|$ will denote its geometric realization, i.e.,
\begin{align*}
    |\bs X| \coloneqq \bigcup_{i=1}^m |X_i|. 
\end{align*}


For an $x\in\Rd$, $\|x\|_2$ and $\|x\|_\infty$ denote the $\ell_2$ and $\ell_\infty$ norms respectively.
We denote the ball in terms of the $\ell_2$ norm centered at $x\in \Rd$ of radius $\varepsilon$ by $B_\varepsilon (x)$. We denote the set of nonnegative reals by $\rpl$ and the set of nonnegative integers by $\zpl$. 

\section{More Definitions and Fundamentals}
\seceq
In addition to the standard definitions given in the previous section, here, we introduce some new concepts that we will use to study contractible cubical sets.
\begin{definition}[Regularity]
    Given a subset $\mathscr{A}$ of $\Rd$, a point $x\in \mathscr A$ is called a \emph{regular point of $\mathscr A$}, if and only if it satisfies at least one of the following conditions. 
    \begin{enumerate}
        \item It is an interior point of $\mathscr A$, i.e., $\exists \varepsilon >0$ such that
        \begin{align*}
            B_\varepsilon (x) \subset \mathscr A.
        \end{align*}
        \item 
            The set 
            \[
             B_\varepsilon (x) \cap \mathscr A \setminus \{x\}   
            \]
    is \emph{contractible} (homotopy equivalent to a single point) for all $\varepsilon>0$ small enough.
    \end{enumerate}
\end{definition}    

    \begin{definition}
    A cubical set $\bs X$ is called a \emph{regular cubical set} if all $x\in|\bs X|$ is a regular point of $|\bs X|$. 
\end{definition}

\begin{example}
For $d=2$ and $d=3$, some examples of irregular points are illustrated in Figure~\ref{irregpiclabel}.

\begin{figure}[ht]
\centering
\resizebox{12cm}{!}{\begin{tikzpicture}[scale=1, line join=round, line cap=round]

\tikzstyle{visible edge}=[black, very thick]
\tikzstyle{hidden edge}=[black, dotted]

\tikzstyle{shared edge}=[red, ultra thick, dotted]

\coordinate (X1) at (-8,2);
\coordinate (Y1) at (-6,2);
\coordinate (Z1) at (-6,4);
\coordinate (T1) at (-8,4);

\coordinate (X2) at (-6,0);
\coordinate (Y2) at (-4,0);
\coordinate (Z2) at (-4,2);
\coordinate (T2) at (-6,2);

\fill[color=gray!40] (X1)--(Y1)--(Z1)--(T1)--cycle;
\draw[visible edge] (X1)--(Y1)--(Z1)--(T1)--cycle;
\fill[color=gray!40] (X2)--(Y2)--(Z2)--(T2)--cycle;
\draw[visible edge] (X2)--(Y2)--(Z2)--(T2)--cycle;

\fill[red] (-6,2) circle (2.5pt);

\node[anchor=center] (Label) at (-2, 2.5) {Irregular Points};

\draw[->, thick] ($(Y1)$) .. controls (-4, 3) .. (Label);

\def\cube#1#2#3#4{
    \coordinate (A#4) at (#1,#2);
    \coordinate (B#4) at (#1+#3,#2);
    \coordinate (C#4) at (#1+#3,#2+#3);
    \coordinate (D#4) at (#1,#2+#3);
    \coordinate (E#4) at (#1+#3*0.5,#2+#3*0.5);
    \coordinate (F#4) at (#1+#3*1.5,#2+#3*0.5);
    \coordinate (G#4) at (#1+#3*1.5,#2+#3*1.5);
    \coordinate (H#4) at (#1+#3*0.5,#2+#3*1.5);
}

\cube{0}{0}{2}{1}

\cube{2}{2}{2}{2}

\foreach \i in {1,2} {
    \shade[left color=gray!10, right color=gray!50, opacity=0.7] (B\i)--(F\i)--(G\i)--(C\i)--cycle;
    \shade[top color=gray!10, bottom color=gray!50, opacity=0.7] (C\i)--(D\i)--(H\i)--(G\i)--cycle;
    
    \draw[visible edge] (A\i)--(B\i)--(C\i)--(D\i)--cycle; 

    \ifthenelse{\i=1}{
    \draw[visible edge] (C\i)--(B\i)--(F\i)--(G\i);}{\draw[visible edge] (C\i)--(B\i)--(F\i)--(G\i)--cycle;} 
    
    \draw[visible edge] (D\i)--(H\i)--(G\i); 
    
    \draw[hidden edge] (A\i)--(E\i)--(F\i);
    \draw[hidden edge] (E\i)--(H\i);
}

\shade[left color=gray!30, right color=gray!70, opacity=0.7] (A1)--(B1)--(C1)--(D1)--cycle;
\shade[left color=gray!30, right color=gray!70, opacity=0.7] (A2)--(B2)--(C2)--(D2)--cycle;

\draw[shared edge] (G1)--(A2);

\draw[->, thick] ($(G1)!0.5!(A2)$) .. controls (1, 4) and (-0.5, 3) .. (Label);

\end{tikzpicture}}
\caption{Illustration of irregular points for $d=2$ and $d=3$}
\label{irregpiclabel}
\end{figure}

\end{example}

\begin{remark}\label{reglocrem}
    Note that regularity is a \emph{local property}. That is, for an $x\in \mathscr A\cup \mathscr B$ with $\mathscr A,\mathscr B\subset \Rd$, if
    \begin{align} \label{regloceqn}
        B_\varepsilon(x) \cap \mathscr A = B_\varepsilon(x) \cap \mathscr B
    \end{align}
    for all $\varepsilon>0$ small enough, then
    $x$ is a regular point of $\mathscr A$ if and only if it is  a regular point of $\mathscr B$. 
\end{remark}

\begin{definition}[\rev{Community}]
    Given a cubical set $\bs X$ and $X\in\mathcal{K}^d$, we call the cubical set
    \begin{align*}
        \mc N_{\bs X}(X)\coloneqq X\cup \left\{\bigcup_{\stackrel{Y\in \bs X}{|Y|\cap |X| \neq \varnothing}} Y\right\}
    \end{align*}
    the community of $X$ in $\bs X$. We will also sometimes refer to the geometric realization $|\mc N_{\bs X}(X)|$ as the community when the context is clear. 
\end{definition}
\begin{remark}
    Note that 
    \[
        |\mc N_{\bs X} (X)| =  |X| \cup \left\{y\in |\bs X| : \|y-\ctr(X)\|_{\infty} \leq  \frac{3}{2} \right\}.
    \]
\end{remark}

\begin{definition}
    The \emph{\rev{outer perimeter}} of the community $|\mc N_{\bs X} (X)|$ is defined as the set
     \[ \left\{y\in |\mc N_{\bs X} (X)| : \|y-\ctr(X)\|_{\infty} = \frac{3}{2} \right\}. \]
\end{definition}
See the figure below for an illustration of the community and its outer parameter when $d=2$.
\begin{figure}[ht] 
\centering
\begin{tikzpicture}[scale=1, line join=round, line cap=round]

\foreach \y in {0, 1, 2}
{
    \draw[very thick, fill=gray!40] (0,\y) rectangle ++(1,1);
}
\draw[very thick, fill=gray!40] (1,0) rectangle ++(1,1);

\draw[very thick, fill=gray!40] (1,1) rectangle ++(1,1);

\node[anchor=center] (Label) at (1.5, 1.5) {$X$};

\draw[line width=3 pt,blue] (1,3)--(0,3);
\draw[line width=3 pt,blue] (0,3)--(0,0);
\draw[line width=3 pt,blue] (0,0)--(2,0);

\end{tikzpicture}
\caption{The cubes that compose the community of $X$ and its outer perimeter shown as thick blue line}  
\label{outerperimfig}
\end{figure}

Next, we give our first observation regarding irregularity, which shows that irregular points of a community should be ``inside'' of it. 
\begin{lemma}\label{irregperimlem}
  For any cubical set $\bs X$ and $X\in\mathcal{K}^d$, the community $|\mc N_{\bs X}(X)|$ cannot have an irregular point on its outer perimeter.    
\end{lemma}

\begin{proof}
    Assume $x=(x_1,\ldots, x_d)$ is an irregular point of $|\mc N_{\bs X} (X)|$ on its outer perimeter. 
     Also assume, without loss of generality, that $x$ is on the hyperplane $\left\{z+\frac{1}{2}\right\}\times \mathbb{R}^{d-1}$ for some $z\in\mathbb Z$, and that 
    for all $\varepsilon$ small enough,
     \begin{align*}
         (x_1, x_1+\varepsilon]\times \{x_2\}\times\ldots\times \{x_d\}\subset |\mc N_{\bs X}(X)|, 
     \end{align*}
      while
     \begin{align*}
         [x_1-\varepsilon, x_1)\times \{x_2\}\times\ldots\times \{x_d\}
     \end{align*}
does not intersect $|\mc N_{\bs X}(X)|$. We observe that for every $d$-dimensional cube $Y\in \mc N_{\bs X}(X)$ with $x\in |Y|$, 
\begin{align*}
    (x_1, x_1+\varepsilon]\times \{x_2\}\times\ldots\times \{x_d\}\subset |Y|,
\end{align*}
and therefore, $B_\varepsilon(x)\cap |Y|\setminus\{x\}$ has a deformation retraction to $(x_1, x_1+\varepsilon]\times \{x_2\}\times\ldots\times \{x_d\}$. 
Since
\begin{align*}
    B_\varepsilon(x)\cap |\mc N_{\bs X}(X)| \setminus\{x\} = \bigcup_{Y\in \mc N_{\bs X}(X)} B_\varepsilon(x)\cap |Y|\setminus\{x\},
\end{align*}
and each set that composes the union on the right hand side has a deformation retraction to  
$(x_1, x_1+\varepsilon]\times \{x_2\}\times\ldots\times \{x_d\}$, which is contractible, we conclude $B_\varepsilon(x)\cap |\mc N_{\bs X}(X)| \setminus\{x\}$ has to be contractible. This, however, is in contradiction with our initial assumption that $x$ was an irregular point of $|\mc N_{\bs X}(X)|$, concluding the proof.
\end{proof}

We establish in the lemma below that regularity of a cubical set is equivalent to the \emph{local} regularity of the communities of its cubes.
\begin{lemma}\label{regiffneighlem}
 The cubical set $\bs X$ is regular if and only if $\mc N_{\bs X}(X)$ is so for all $X\in\bs X$.    
\end{lemma}
\begin{proof}
    Assume $x\in\Rd$ is an irregular point of $|\mc N_{\bs X}(X)|$. This requires $x$ be not in the outer perimeter of 
    $|\mc N_{\bs X}(X)|$ due to Lemma~\ref{irregperimlem}. For all such $x$, though, \eqref{regloceqn} holds for $\mathscr A = |\mc N_{\bs X}(X)|$ and $\mathscr B=|\bs X|$, and therefore $x$ is an irregular point of $|\bs X|$ as well, which implies $\bs X$ is irregular.  
        
    Assume now that $x$ is an irregular point of $|\bs X|$. Choose an arbitrary $X\in\bs X$ such that $x\in |X|$. Using Remark~\ref{reglocrem} again for the same choice of $\mathscr{A}$ and $\mathscr B$ as above, we conclude 
     $x$ is also an irregular point of $|\mc N_{\bs X}(X)|$. We proved if $x$ is an irregular point of $|\mc N_{\bs X}(X)|$ then it is also an irregular point of $|\bs X|$, and if $x$ is an irregular point of $\bs X$ then there exists $X\in \bs X$ such that it is irregular in $|\mc N_{\bs X}(X)|$. These two assertions prove the claim in the lemma.
\end{proof}

Our last lemma in this section will be used in the next section where we will define the objects of main interest in this paper.

\begin{lemma}\label{commcontractlem}
    For any cubical set $\bs X$ and $X\in \bs X$, $|\mc N_{\bs X}(X)|$ is contractible. 
\end{lemma}
\begin{proof}
Note the contractibility of realization of the cube complex $|\mc N_{\bs X}(X)|$, as a CW-complex, is equivalent to the existence of a deformation retraction to a point (see Proposition 2.5 of \cite{milnorcons}). A deformation retraction from $|\mc N_{\bs X}(X)|$ to any point of $|X|$ can be defined trivially.    
\end{proof}



\section{Clumps}
\seceq
We define below the main object of study in this paper, and the partial order on them underlying the Markov process we will define.

\begin{definition}[Clump]
A cubical set $\bs X$ is called a \emph{clump} if it is regular and its geometric realization $|\bs X|$ is contractible.
\end{definition}


\begin{definition}[\rev{Indent action and $k$-Indentability}]
A clump $\bs X$ is called \emph{$k$-indentable} if there exist distinct cubes $X_1,\ldots,X_k \in\bs X$ such that for each $i\in \{1,\ldots,k\}$, $\bs X\setminus X_i$ is a clump, and there exists a strong deformation retraction $f_i:|\bs X|\times [0,1]\to |\bs X|$ from $|\bs X|$ to $|\bs X\setminus X_i|$ satisfying $f(x,t)\in |X_i|$ for all $x\in|X_i|$ and $t\in[0,1)$. We say that $\bs X$ is indentable ``to $\bs X\setminus X_i$'' or ``through $X_i$'' for any $i\in \{1,\ldots,k\}$, and call the set $\{X_1,\ldots,X_k\}$ its \emph{indent set}. We call 1-indentable clumps, sometimes, briefly, indentable. 
\end{definition}
Note that, obviously, if $\bs X$ is $k$-indentable then it is $\ell$-indentable for all $\ell\leq k$.
\begin{figure}[ht] 
\centering
\begin{tikzpicture}[scale=1, line join=round, line cap=round]

\foreach \x in {0,5}
{
\foreach \y in {0, 1, 2}
{
    \draw[very thick, fill=gray!40] (\x,\y) rectangle ++(1,1);
}
\draw[very thick, fill=gray!40] (\x+1,0) rectangle ++(1,1);
}

\draw[very thick, fill=gray!40] (1,1) rectangle ++(1,1);

\draw[thick,-stealth](3,1.5) -- (4,1.5);

\node[anchor=center] (Label) at (1.5, 1.5) {$X$};

\end{tikzpicture}
\caption{Indent action for a cubical set in $\mathbb{R}^2$}  
\label{indentfig}
\end{figure}

The following lemma establishes that \emph{indentability} through a cube is a \emph{community} property. This will make sure that the rules of actions of our dynamic model for cubical sets that we will define later on will be local. 
\begin{lemma}\label{neighbcollem}
 The clump $\bs X$ is indentable through some $X\in\bs X$ if and only if $\mc N_{\bs X}(X)$ is a  clump \rev{indentable through $X$}. 
\end{lemma}
\begin{proof}
We first note that a strong deformation retraction $f$ from $|\mc N_{\bs X}(X)|$ to $|\mc N_{\bs X}(X)\setminus X|$ such that $f(x,t)\in |X_i|$ for all $x\in|X|$ and $t\in[0,1)$ trivially defines a strong deformation retraction from $|\bs X|$ to $|\bs X\setminus X|$ with the same property, and vice versa. 
    To prove the ``if'' part, assume for an $X\in \bs X$, $\mc N_{\bs X}(X)$ is an indentable clump through $X$, which implies $\mc N_{\bs X}(X)\setminus X$ is regular.  We will show that $\bs X\setminus X$ is regular. For that, take a point $x\in |\bs X\setminus X|$. If $x\in |X|$,
    using Remark~\ref{reglocrem} with the choices of $\mathscr A = |\bs X\setminus X|$ and $\mathscr B = |\mc N_{\bs X}(X)\setminus X|$, and that $\mc N_{\bs X}(X)\setminus X$ is regular, $x$ has to be regular in $|\bs X\setminus X|$.    
    If, on the other hand, $x\notin |X|$, we use Remark~\ref{reglocrem} with $\mathscr A = |\bs X\setminus X|$, $\mathscr B = |\bs X|$, and the fact that $\bs X$ was regular since it was a clump, to obtain that $x$ is regular in $|\bs X\setminus X|$.  
     Together with the existence of a strong deformation retraction, this establishes that $\bs X\setminus X$ is a clump, and $\bs X$ is indentable through $X$.

     To prove the reverse implication, assume $\bs X$ is an indentable clump through some $X\in\bs X$. For the regularity of $\mc N_{\bs X} (X)\setminus X$, take an $x\in |\mc N_{\bs X} (X)\setminus X|$. If $x$ is on the outer perimeter of $|\mc N_{\bs X} (X)\setminus X|$, then $x$ is regular in $|\mc N_{\bs X} (X)\setminus X|$ due to Lemma~\ref{irregperimlem}. If otherwise $x$ is not on the outer perimeter of $|\mc N_{\bs X} (X)\setminus X|$, we invoke Remark~\ref{reglocrem} with $\mathscr{A} = |\bs X\setminus X|$ and $\mathscr{B} = |\mc N_{\bs X} (X)\setminus X|$ to conclude $x$ is regular in $|\mc N_{\bs X} (X)\setminus X|$, which proves that $\mc N_{\bs X} (X)\setminus X$ is regular. Repeating the argument with some $x\in |\mc N_{\bs X}(X)|$ with $\mathscr{A} = |\bs X|$ and $\mathscr{B} = |\mc N_{\bs X} (X)|$ proves that $\mc N_{\bs X} (X)$ is also regular. 
     Combining this with the contractibility of $|\mc N_{\bs X} (X)|$ due to  Lemma~\ref{commcontractlem}, and the existence of strong deformation retraction from $|\mc N_{\bs X} (X)|$ to $|\mc N_{\bs X} (X)\setminus X|$ establishes that $\mc N_{\bs X} (X)$ is an indentable clump through $X$.    
\end{proof}

Below, we define the inverse of the indent action, expansion.
\begin{definition}[Expand]\label{expanddef}
Given a cubical set $\bs X$, $\card(\bs X)\geq 1$ and a cube $Y\in \mathcal{K}^d\setminus \bs X$, we call the cubical set $\bs X\cup Y$ 
the \emph{expansion} of $\bs X$ \emph{through $Y$}.
A clump $\bs X$ is called \emph{expandable} through $Y\in \mathcal{K}^d\setminus \bs X$ if $\bs X\cup Y$ is indentable through $Y$.
\end{definition}

Due to Lemma~\ref{neighbcollem} and by Definition~\ref{expanddef} we have the following.
\begin{corollary}
\label{neighbexplem}
 The clump $\bs X$ is expandable through some $Y\in\mathcal{K}^d$ if and only if $\mc N_{\bs X}(Y)$ is a clump indentable through $Y$ and $Y\notin\bs X$. 
\end{corollary}

 Next, we give definitions that we will use in our study of cubical sets when $d=2$. We establish the definitions in any $d$, for the sake of generality, nevertheless.  
\begin{definition}[Paths and Loops]\label{pathdef}
    A sequence of distinct $d$-dimensional elementary cubes $P_1,\ldots, P_n$ is called a \emph{path (of length $n\geq 2$)} if $P_i\cap P_{i+1} \in \mathcal{K}^{d-1}$ for all $1\leq i\leq n-1$. The path $P_1,\ldots, P_n$ is a \emph{loop} if $P_1\cap P_n \in \mathcal{K}^{d-1}$ and $n\geq 3$. The term ``path'' will sometimes refer to the cubical set $\bs P\coloneqq \{P_1,\ldots ,P_n\}$   
    when there is no room for confusion. We say a path $\bs P$ is in the cubical set $\bs X$ if $\bs P\subseteq \bs X$.
\end{definition}

\begin{definition}[Scaffold]
    Given a path $\bs P = \{P_1,\ldots,P_n\}$, we call the set
    \[ \scf(\bs P) \coloneqq 
    \begin{cases}
        \left\{\bigcup_{i=1}^{n-1}  \overbar{\ctr(P_i) \ctr(P_{i+1})}\right\} \cup \overbar{\ctr(P_n) \ctr(P_{1})} \quad &\text{if $\bs P$ is a loop}, \\
        \bigcup_{i=1}^{n-1}  \overbar{\ctr(P_i) \ctr(P_{i+1})} &\text{otherwise.}
    \end{cases}
       \]
    the \emph{scaffold} of $\bs P$, where $\overbar{ab}\subset\Rd$ denotes the closed line segment connecting $a,b\in \Rd$.  
\end{definition}
We have a similar definition for general cubical sets as follows.

\begin{definition}[Skeleton]\label{defSkel}
    Given a cubical set $\bs X$ we call the set
     \[ \left\{\bigcup_{X\in \bs X} \ctr(X)\right\} \cup \left\{ \bigcup_{\stackrel{X,Y\in\bs X}{X\cap Y \in \mathcal{K}^{d-1} }} \overbar{\ctr(X) \ctr(Y)} \right\}  \]
    the \emph{skeleton} of $\bs X$.
 \end{definition}

We next define a concept only for $d=2$ case, to be used in the next section.
\begin{definition}[\rev{Inland}, $d=2$]
Given a loop $\bs L$, composed of subsequent distinct cubes $L_1,\ldots L_n$, we define its \emph{inland}, $\inl(\bs L)\subset \mathbb{R}^2$, as the union of $|\bs L|$ and the subset of $\mathbb{R}^2$ bounded by its scaffold.  
\end{definition}

Now that we have established the fundamentals of \emph{clumps} and actions defined on them, we are ready to focus on the planar special case $(d=2)$, where we can prove some stronger statements.

\section{Planar Clumps}
\seceq
We state our first main theorem below.

\begin{theorem}\label{indentthm}
    All planar clumps are 2-indentable.
\end{theorem}

\begin{remark}\label{lineclumprem}
 Note that 2-indentability is the strongest notion of indentability that can be proven for all clumps of all sizes on the plane. The ``straight line'' clump shown in the Figure~\ref{lineclumpfig} below can be made arbitrarily large, but although it is 2-indentable (through first and last cubes), it is never 3-indentable.
 \begin{figure}[ht] 
 \centering
\begin{tikzpicture}[scale=1, line join=round, line cap=round]

\foreach \x in {0, 1, 2, 3}
    {
        \draw[very thick, fill=gray!40] (\x,0) rectangle ++(1,1);
    }

\fill[black] (4.5,0.5) circle (1pt);
\fill[black] (4.65,0.5) circle (1pt);
\fill[black] (4.8,0.5) circle (1pt);

\fill[black] (-0.5,0.5) circle (1pt);
\fill[black] (-0.65,0.5) circle (1pt);
\fill[black] (-0.8,0.5) circle (1pt);

\end{tikzpicture}
\caption{Line clump.}  
\label{lineclumpfig}
\end{figure}

\end{remark}

Before giving the proof of the theorem, we will give some auxiliary lemmas. Unless otherwise noted, we assume $d=2$ for the rest of this section. We start with  definitions of some specific cubical sets that will play an important role in our proof of Theorem~\ref{indentthm}. Note that for the rest of this section, we will continue calling our building blocks ``cubes'' in order to be consistent with the general notation, although one might argue calling them ``squares'' would have been more intuitive.

\begin{definition}[Rectangle Configuration]
    We define the \emph{2-by-3 rectangle configuration} as the cubical set $\bs R_6$ composed of 6 cubes (illustrated in Figure~\ref{rectfig}). More concretely,
    \begin{align*}
    \begin{split}
    \bs R_6 \coloneqq &\bigg\{\left[-\frac{3}{2}, -\frac{1}{2}\right]\times \left[\frac{1}{2}, \frac{3}{2}\right], \left[-\frac{1}{2}, \frac{1}{2}\right]\times \left[\frac{1}{2}, \frac{3}{2}\right],
    \left[\frac{1}{2}, \frac{3}{2}\right]\times \left[\frac{1}{2}, \frac{3}{2}\right] , \\
    &\left[-\frac{3}{2}, -\frac{1}{2}\right]\times \left[-\frac{1}{2}, \frac{1}{2}\right] ,\left[-\frac{1}{2}, \frac{1}{2}\right]\times \left[-\frac{1}{2}, \frac{1}{2}\right] ,
    \left[\frac{1}{2}, \frac{3}{2}\right]\times \left[-\frac{1}{2}, \frac{1}{2}\right]\bigg\}.
    \end{split}
    \end{align*}
    
\begin{figure}[ht]    
\centering
\begin{tikzpicture}[scale=1, line join=round, line cap=round]

\foreach \x in {0, 1, 2}
    \foreach \y in {0, 1}
    {
        \draw[very thick, fill=gray!40] (\x,\y) rectangle ++(1,1);
    }

\fill[black] (1.5,0.5) circle (1pt);
\node[anchor=north] (Label) at (1.5, 0.55) {$(0,0)$};

\end{tikzpicture}
\caption{2-by-3 rectangle configuration $\bs R_6$ with the origin of $\bb R^2$ for reference.}  
\label{rectfig}
\end{figure}

We call the 90-degree rotational symmetry of $\bs R_6$, defined in a straightforward way, $\bs R_6^\top$, which is illustrated in Figure~\ref{rectfigtrans}.
\begin{figure}[ht]  
\centering
\begin{tikzpicture}[scale=1, line join=round, line cap=round]

\foreach \x in {0, 1}
    \foreach \y in {0, 1, 2}
    {
        \draw[very thick, fill=gray!40] (\x,\y) rectangle ++(1,1);
    }

\fill[black] (1.5,1.5) circle (1pt);
\node[anchor=north] (Label) at (1.5, 1.55) {$(0,0)$};

\end{tikzpicture}
\caption{Rotational symmetry of $\bs R_6$, i.e., $\bs R_6^{\top}$}  
\label{rectfigtrans}
\end{figure}

We say that a cubical set $\bs X$ includes (or does have) a 2-by-3 rectangle (at $z$) if there exists $z\in \mathbb Z^2$ such that 
\[\bs X \supseteq \bs R_6 + z \]
or 
\[\bs X  \supseteq  \bs R_6^{\top} + z , \]
where $\bs X + z$ is the translation of $\bs X$ defined in the natural way. 
    
\end{definition}

\begin{definition}[Double-deck Configuration]
The cubical set \emph{double-deck configuration} is shown in Figure~\ref{deckfig}. A formal definition is straightforward and omitted to save space.
\begin{figure}[ht] 
\centering
\begin{tikzpicture}[scale=1, line join=round, line cap=round]

\foreach \x in {-1, 0, 1, 2}
    {
        \draw[very thick, fill=gray!40] (\x-0.5,-0.5) rectangle ++(1,1);
    }

\foreach \x in {0, 1}
    {
        \draw[very thick, fill=gray!40] (\x-0.5,0.5) rectangle ++(1,1);
    }

\fill[black] (0,0) circle (1pt);
\node[anchor=north] (Label) at (0, 0.05) {$(0,0)$};

\end{tikzpicture}
\caption{Double-deck configuration at the origin (0,0).} 
\label{deckfig}
\end{figure}

   We adopt the same terminology we introduced previously for the rectangular configuration. That is, we say a cubical set \emph{has a double-deck (configuration)} if it includes a translation or rotation of a double deck at $(0,0)$.
\end{definition}

\begin{definition}[Human Configuration]
 The cubical set \emph{Human configuration}, named after its resemblance to the silhouette of the human body, is shown in Figure~\ref{Humanfig}. The cross signs in the figure indicate that for a subset of a cubical set to be counted as Human, the corresponding locations should be empty in the cubical set. We use the same terminology regarding the inclusion of a Human configuration as we described for previous configurations. Note that, in contrast to the rectangle and double-deck configurations, Human also requires some locations outside of its borders be unoccupied. Furthermore, Human is 1-indentable since the community of the cube centered at $(0,-2)$ is indentable, \rev{which is implied by that there is no cube at $(0,-3)$.}  
 \begin{figure}[ht]
 \centering
\begin{tikzpicture}[scale=1, line join=round, line cap=round]

\foreach \x in {0, 1, 2}
    \foreach \y in {0, 1}
    {
        \draw[very thick, fill=gray!40] (\x,\y) rectangle ++(1,1);
    }

\foreach \x in {-1, 3}
    {
        \draw[very thick, fill=gray!40] (\x,1) rectangle ++(1,1);
    }

\foreach \y in {2, 3}
    {
        \draw[very thick, fill=gray!40] (1,\y) rectangle ++(1,1);
    }

\foreach \x in {0, 2}
    {
        \draw[very thick, fill=gray!40] (\x,-1) rectangle ++(1,1);
    }

\fill[black] (1.5,2.5) circle (1pt);
\node[anchor=north] (Label) at (1.5, 3.09) {$(0,0)$};
\fill[black] (1.5,0.5) circle (1pt);
\node[anchor=north] (Label) at (1.5, 1.09) {$(0,\hspace{-0.11em}-2)$};
\draw[very thick] (0,2) -- (1,3);
\draw[very thick] (0,3) -- (1,2);

\draw[very thick] (2,2) -- (3,3);
\draw[very thick] (2,3) -- (3,2);

\draw[very thick] (1,-1) -- (2,0);
\draw[very thick] (1,0) -- (2,-1);

\end{tikzpicture}
\caption{Human configuration at $(0,0)$.}  
\label{Humanfig}
\end{figure}
\end{definition}

Our proof of Theorem~\ref{indentthm} is based on the following idea. We can roughly categorize the clumps as the ``branchy'' ones and ``round'' ones. For round ones, 2-indentability is ensured by the fact that some cubes on the border of the clump will have indentable communities since they can neither be ``bottlenecks'' that keeps the connectivity nor be ``embedded deep'' whose removal would introduce a nontrivial 1-dimensional homology. The 2-indentability of branchy ones, on the other hand, follows from arguing that the longest path (branch), looks like a ``line'' and therefore its first and last cubes have indentable communities (see Remark~\ref{lineclumprem}). 

Now we state the series of lemmas that connect 2-indentability and the existence of special configurations defined above in a cubical set. The results until Lemma~\ref{norectlem} will be related to the 2-indentability of ``round'' clumps.

\begin{lemma}\label{ddeckcollpos}
    If a clump has a double-deck at the origin then its indent set includes a cube with the center $(x,y)$ where $y\geq 1$.
\end{lemma}
\begin{proof}
   Assume a clump $\bs X$ has a double-deck at the origin and it does not have any cube centered at a positive $y$-coordinate in its indent set. We label the cubes of $\bs X$ in and around the configuration as in Figure~\ref{dbldeckprf}. Note that we do not yet know whether the nearby cubes are in the clump or not, hence they are drawn with lighter shade and bordered with dashed edges. 
   
   First, we notice that unless $X_8\in \bs X$, $\mc N_{\bs X} (X_5)$ would be indentable through $X_5$ no matter whether other cubes of $\mc N_{\bs X} (X_5)$ are in $\bs X$ or not, and since $X_5$ has center at $(0,1)$, this would contradict our assumption, therefore $X_8\in \bs X$. Similarly $X_9$ is also in $\bs X$ because otherwise $\mc N_{\bs X} (X_6)$ would be indentable. Given that now we know $\{X_8,X_9\}\subset X$, looking back to $\mc N_{\bs X} (X_5)$ again, we see that $\mc N_{\bs X} (X_5)$ would be indentable unless $X_7\in \bs X$. Repeating the same argument for $\mc N_{\bs X} (X_6)$, we conclude $\{X_7,X_{10}\}\subset \bs X$. Now the cubical set $\bs X$ looks as in Figure~\ref{dbldeckprf}-ii. Repeating the steps now for the double deck composed of cubes $\{X_7,X_5,X_6,X_{10}, X_8, X_9\}$, we argue that $\{X_{11},X_{12},X_{13},X_{14}\}\subset \bs X$. This configuration in $\bs X$ can be stacked up indefinitely, which would contradict that $\bs X$ is a finite set since it is a clump. Therefore we conclude the proof of the statement in the lemma. 
 \begin{figure}[ht]
 \centering
\begin{tikzpicture}[scale=1, line join=round, line cap=round]

\tikzmath{
\sepax = 6;}

\foreach \x in {-1, 0, 1, 2}
    {
        \draw[very thick, fill=gray!40] (\x-0.5,-0.5) rectangle ++(1,1);
    }

\foreach \x in {0, 1}
    {
        \draw[very thick, fill=gray!40] (\x-0.5,0.5) rectangle ++(1,1);
    }

\draw[very thick, fill=gray!10, dashed] (-1.5,0.5) rectangle ++(1,1);
\draw[very thick, fill=gray!10, dashed] (-0.5,1.5) rectangle ++(1,1);
\draw[very thick, fill=gray!10, dashed] (0.5,1.5) rectangle ++(1,1);
\draw[very thick, fill=gray!10, dashed] (1.5,0.5) rectangle ++(1,1);

\foreach \x in {-1, 0, 1, 2}
    {
        \draw[very thick, fill=gray!40] (\x-0.5+\sepax,-0.5) rectangle ++(1,1);
    }

\foreach \x in {0, 1}
    {
        \draw[very thick, fill=gray!40] (\x-0.5+\sepax,0.5) rectangle ++(1,1);
    }

\draw[very thick, fill=gray!40] (-1.5+\sepax,0.5) rectangle ++(1,1);
\draw[very thick, fill=gray!40] (-0.5+\sepax,1.5) rectangle ++(1,1);
\draw[very thick, fill=gray!40] (0.5+\sepax,1.5) rectangle ++(1,1);
\draw[very thick, fill=gray!40] (1.5+\sepax,0.5) rectangle ++(1,1);

\foreach \x in {0,\sepax}
{
\node[anchor=center] (Label) at (-1+\x, 0) {$X_1$};
\node[anchor=center] (Label) at (0+\x, 0) {$X_2$};
\node[anchor=center] (Label) at (1+\x, 0) {$X_3$};
\node[anchor=center] (Label) at (2+\x, 0) {$X_4$};
\node[anchor=center] (Label) at (0+\x, 1) {$X_5$};
\node[anchor=center] (Label) at (1+\x, 1) {$X_6$};

\node[anchor=center] (Label) at (-1+\x, 1) {$X_7$};
\node[anchor=center] (Label) at (0+\x, 2) {$X_8$};
\node[anchor=center] (Label) at (1+\x, 2) {$X_9$};
\node[anchor=center] (Label) at (2+\x, 1) {$X_{10}$};
}

\node[anchor=center] (Label) at (-1.5, 3.8) {\textbf{i)}};

\node[anchor=center] (Label) at (-1.5+\sepax, 3.8) {\textbf{ii)}};

\draw[very thick, fill=gray!10, dashed] (-1.5+\sepax,1.5) rectangle ++(1,1);
\draw[very thick, fill=gray!10, dashed] (-0.5+\sepax,2.5) rectangle ++(1,1);
\draw[very thick, fill=gray!10, dashed] (0.5+\sepax,2.5) rectangle ++(1,1);
\draw[very thick, fill=gray!10, dashed] (1.5+\sepax,1.5) rectangle ++(1,1);
\node[anchor=center] (Label) at (-1+\sepax, 2) {$X_{11}$};
\node[anchor=center] (Label) at (0+\sepax, 3) {$X_{12}$};
\node[anchor=center] (Label) at (1+\sepax, 3) {$X_{13}$};
\node[anchor=center] (Label) at (2+\sepax, 2) {$X_{14}$};

\end{tikzpicture}
\caption{Double-deck configuration in a cubical set $\bs X$. Note that $\ctr(X_2) = (0,0)$} 
\label{dbldeckprf}
\end{figure}
\end{proof}

\begin{corollary}\label{ddeckcor}
    If a clump has two double decks back-to-back (or any translation and rotation symmetry of it) as shown in Figure~\ref{dbldbldeckfig}, then it is 2-indentable.

\begin{figure}[ht] 
\centering
\begin{tikzpicture}[scale=1, line join=round, line cap=round]

\foreach \x in {-1, 0, 1, 2}
\foreach \y in {-0.5,-1.5}
    {
        \draw[very thick, fill=gray!40] (\x-0.5,\y) rectangle ++(1,1);
    }

\foreach \x in {0, 1}
\foreach \y in {0.5,-2.5}
    {
        \draw[very thick, fill=gray!40] (\x-0.5,\y) rectangle ++(1,1);
    }

\fill[black] (0,0) circle (1pt);
\node[anchor=north] (Label) at (0, 0.05) {$(0,0)$};

\end{tikzpicture}
\caption{Back-to-back double-deck configurations.} 
\label{dbldbldeckfig}
\end{figure}
    
\end{corollary}
\begin{proof}
Follows by the fact that a clump that has back-to-back double-decks must necessarily have at least two cubes, one with positive and another with negative $y$-coordinate centers, in its indent set due to Lemma~\ref{ddeckcollpos}.    
\end{proof}

\begin{lemma}\label{lem ab}
    If a clump includes a 2-by-3 rectangle then it is \allowbreak 2-indentable or it has a Human. 
\end{lemma}

\begin{proof}
Assume that a clump $\bs X$
has $\bs R_6$ and it is not 2-indentable. We will show that it does have a Human. As in Figure~\ref{configaenum}-i), let us name the cubes in $\bs R_6$ (bordered with \rev{thick} blue line) as $X_1,\ldots,X_6$ in clockwise order and the nearby cubes that are potentially included in the clump. 
\begin{figure}[ht] 
\centering
\begin{tikzpicture}[scale=1, line join=round, line cap=round]

\tikzmath{
\sepa = 7;
\sepay = 6.5;}

\foreach \x in {1, 2,3}
    \foreach \y in {0, 1}
    {
        \draw[very thick, fill=gray!40] (\x,\y) rectangle ++(1,1);
    }

\foreach \x in {0, 4}
\foreach \y in {0, 1}
    {
        \draw[very thick, fill=gray!10,dashed] (\x,\y) rectangle ++(1,1);
    }

\foreach \x in {1,2,3}
\foreach \y in {-1,2}
    {
        \draw[very thick, fill=gray!10,dashed] (\x,\y) rectangle ++(1,1);
    }

\draw[line width=3 pt,blue] (1,0) rectangle ++(3,2); 

\foreach \x in {1, 2,3}
    \foreach \y in {0, 1}
    {
        \draw[very thick, fill=gray!40] (\x+\sepa,\y) rectangle ++(1,1);
    }



\foreach \x in {0, 4}
\foreach \y in {0, 1}
    {
        \draw[very thick, fill=gray!10,dashed] (\x+\sepa,\y) rectangle ++(1,1);
    }

    \draw[very thick, fill=gray!40] (2+\sepa,2) rectangle ++(1,1);

\foreach \x in {1,3}
\foreach \y in {-1,2}
    {
        \draw[very thick, fill=gray!10,dashed] (\x+\sepa,\y) rectangle ++(1,1);
    }

\foreach \x in {0,\sepa}
 \draw[very thick, fill=gray!10,dashed] (2+\x,3) rectangle ++(1,1);

\draw[line width=3 pt,blue] (1+\sepa,0) rectangle ++(3,2);

\foreach \x in {0,\sepa}
{
\node[anchor=center] (Label) at (
        1.5+\x, 1.5) {$X_1$};
\node[anchor=center] (Label) at (
        2.5+\x, 1.5) {$X_2$};
\node[anchor=center] (Label) at (
        3.5+\x, 1.5) {$X_3$};
\node[anchor=center] (Label) at (
        3.5+\x, 0.5) {$X_4$};
\node[anchor=center] (Label) at (
        2.5+\x, 0.5) {$X_5$};
\node[anchor=center] (Label) at (
        1.5+\x, 0.5) {$X_6$};

\node[anchor=center] (Label) at (
        0.5+\x, 1.5) {$X_7$};
\node[anchor=center] (Label) at (
        1.5+\x, 2.5) {$X_8$};
\node[anchor=center] (Label) at (
        2.5+\x, 2.5) {$X_9$};
\node[anchor=center] (Label) at (
        3.5+\x, 2.5) {$X_{10}$};
\node[anchor=center] (Label) at (
        4.5+\x, 1.5) {$X_{11}$};
\node[anchor=center] (Label) at (
        4.5+\x, 0.5) {$X_{12}$};
\node[anchor=center] (Label) at (
        3.5+\x, -0.5) {$X_{13}$};
\node[anchor=center] (Label) at (1.5+\x, -0.5) {$X_{15}$};

\node[anchor=center] (Label) at (
        0.5+\x, 0.5) {$X_{16}$};
\node[anchor=center] (Label) at (2.5+\x, 3.5) {$X_{17}$};
}

\node[anchor=center] (Label) at (2.5, -0.5) {$X_{14}$};

\node[anchor=center] (Label) at (0, 3.5) {\textbf{i)}};

\node[anchor=center] (Label) at (0+\sepa, 3.5) {\textbf{ii)}};

\draw[very thick] (2+\sepa,-1) -- (3+\sepa,0); 
\draw[very thick] (2+\sepa,-0) -- (3+\sepa,-1);

\foreach \x in {1, 2,3}
    \foreach \y in {0, 1}
    {
        \draw[very thick, fill=gray!40] (\x,\y-\sepay) rectangle ++(1,1);
    }
    
\foreach \x in {0, 4}
    {
        \draw[very thick, fill=gray!10,dashed] (\x,0-\sepay) rectangle ++(1,1);
    }

\foreach \x in {0, 4}
    {
        \draw[very thick, fill=gray!40] (\x,1-\sepay) rectangle ++(1,1);
    }

    \draw[very thick, fill=gray!40] (2,2-\sepay) rectangle ++(1,1);

\foreach \x in {1,3}
\foreach \y in {-1,2}
    {
        \draw[very thick, fill=gray!10,dashed] (\x,\y-\sepay) rectangle ++(1,1);
    }

 \draw[very thick, fill=gray!40] (2,3-\sepay) rectangle ++(1,1);

\draw[line width=3 pt,blue] (1,-\sepay) rectangle ++(3,2);
 
\node[anchor=center] (Label) at (
        1.5, 1.5-\sepay) {$X_1$};
\node[anchor=center] (Label) at (
        2.5, 1.5-\sepay) {$X_2$};
\node[anchor=center] (Label) at (
        3.5, 1.5-\sepay) {$X_3$};
\node[anchor=center] (Label) at (
        3.5, 0.5-\sepay) {$X_4$};
\node[anchor=center] (Label) at (
        2.5, 0.5-\sepay) {$X_5$};
\node[anchor=center] (Label) at (
        1.5, 0.5-\sepay) {$X_6$};
\node[anchor=center] (Label) at (
        0.5, 1.5-\sepay) {$X_7$};
\node[anchor=center] (Label) at (
        1.5, 2.5-\sepay) {$X_8$};
\node[anchor=center] (Label) at (
        2.5, 2.5-\sepay) {$X_9$};
\node[anchor=center] (Label) at (
        3.5, 2.5-\sepay) {$X_{10}$};
\node[anchor=center] (Label) at (
        4.5, 1.5-\sepay) {$X_{11}$};
\node[anchor=center] (Label) at (
        4.5, 0.5-\sepay) {$X_{12}$};
\node[anchor=center] (Label) at (
        3.5, -0.5-\sepay) {$X_{13}$};
\node[anchor=center] (Label) at (1.5, -0.5-\sepay) {$X_{15}$};
\node[anchor=center] (Label) at (
        0.5, 0.5-\sepay) {$X_{16}$};
\node[anchor=center] (Label) at (2.5, 3.5-\sepay) {$X_{17}$};
\draw[very thick] (2,-1-\sepay) -- (3,0-\sepay); 
\draw[very thick] (2,-\sepay) -- (3,-1-\sepay);

\node[anchor=center] (Label) at (0, 3.5-\sepay) {\textbf{iii)}};

\foreach \x in {1, 2,3}
    \foreach \y in {0, 1}
    {
        \draw[very thick, fill=gray!40] (\x+\sepa,\y-\sepay) rectangle ++(1,1);
    }
    
\foreach \x in {0, 4}
    {
        \draw[very thick, fill=gray!10,dashed] (\x+\sepa,0-\sepay) rectangle ++(1,1);
    }

\foreach \x in {0, 4}
    {
        \draw[very thick, fill=gray!40] (\x+\sepa,1-\sepay) rectangle ++(1,1);
    }

    \draw[very thick, fill=gray!40] (2+\sepa,2-\sepay) rectangle ++(1,1);

\foreach \x in {1,3}
    {
        \draw[very thick, fill=gray!40] (\x+\sepa,-1-\sepay) rectangle ++(1,1);
    }

 \draw[very thick, fill=gray!40] (2+\sepa,3-\sepay) rectangle ++(1,1);

\draw[line width=3 pt,blue] (1+\sepa,-\sepay) rectangle ++(3,2);
 
\node[anchor=center] (Label) at (
        1.5+\sepa, 1.5-\sepay) {$X_1$};
\node[anchor=center] (Label) at (
        2.5+\sepa, 1.5-\sepay) {$X_2$};
\node[anchor=center] (Label) at (
        3.5+\sepa, 1.5-\sepay) {$X_3$};
\node[anchor=center] (Label) at (
        3.5+\sepa, 0.5-\sepay) {$X_4$};
\node[anchor=center] (Label) at (
        2.5+\sepa, 0.5-\sepay) {$X_5$};
\node[anchor=center] (Label) at (
        1.5+\sepa, 0.5-\sepay) {$X_6$};
\node[anchor=center] (Label) at (
        0.5+\sepa, 1.5-\sepay) {$X_7$};

\node[anchor=center] (Label) at (
        2.5+\sepa, 2.5-\sepay) {$X_9$};

\node[anchor=center] (Label) at (
        4.5+\sepa, 1.5-\sepay) {$X_{11}$};
\node[anchor=center] (Label) at (
        4.5+\sepa, 0.5-\sepay) {$X_{12}$};
\node[anchor=center] (Label) at (
        3.5+\sepa, -0.5-\sepay) {$X_{13}$};
\node[anchor=center] (Label) at (1.5+\sepa, -0.5-\sepay) {$X_{15}$};
\node[anchor=center] (Label) at (
        0.5+\sepa, 0.5-\sepay) {$X_{16}$};
\node[anchor=center] (Label) at (2.5+\sepa, 3.5-\sepay) {$X_{17}$};
\draw[very thick] (2+\sepa,-1-\sepay) -- (3+\sepa,0-\sepay); 
\draw[very thick] (2+\sepa,-\sepay) -- (3+\sepa,-1-\sepay);

\foreach \x in {-1,1}
{
\draw[very thick] (2+\sepa+\x,-1-\sepay+3) -- (3+\sepa+\x,0-\sepay+3); 
\draw[very thick] (2+\sepa+\x,-\sepay+3) -- (3+\sepa+\x,-1-\sepay+3);
}

\draw[very thick] (2+\sepa,-1-\sepay) -- (3+\sepa,0-\sepay); 
\draw[very thick] (2+\sepa,-\sepay) -- (3+\sepa,-1-\sepay);

\node[anchor=center] (Label) at (\sepa, 3.5-\sepay) {\textbf{iv)}};

\end{tikzpicture}
\caption{$\bs X$ with an $\bs R_6$ on i, and with the additional squares it needs to include/exclude so that it is not 2-indentable on ii, iii, and iv.}
\label{configaenum}
\end{figure}

We first argue that at least one of $X_9$ and $X_{14}$ needs to be in $\bs X$ because, otherwise, both $\mc N_{\bs X} (X_2)$ and $\mc N_{\bs X} (X_5)$ would have been indentable no matter whether $X_8$, $X_{10}$, $X_{13}$, and $X_{15}$ are in $\bs X$ or not. However, this would make $\bs X$ 2-indentable due to Lemma~\ref{neighbcollem} leading to a contradiction. We also see that at most one of $X_9$ or $X_{14}$ can be in $\bs X$ due to Corollary~\ref{ddeckcor}, since otherwise $\bs X$ would have a 90-degree rotation symmetry of back-to-back double-decks. Therefore, we conclude that exactly one of $X_9$ or $X_{14}$ must be in $\bs X$. Let us assume $X_9\in\bs X$ and $X_{14}\notin \bs X$ without loss of generality, due to symmetry (shown in Figure~\ref{configaenum}-ii). 

Note that now $\mc N_{\bs X} (X_5)$ is indentable, $\bs X$ cannot have another cube in its indent set. Since the cubes $X_1$, $X_3$, and $X_9$ would be in the indent set of $\bs X$ unless $X_7$, $X_{11}$, and $X_{17}$ are in $\bs X$, we conclude that $\{X_7, X_{11},X_{17}\}\subset \bs X$.
Figure~\ref{configaenum}-iii) shows the cubes we have argued to (not) be in $\bs X$ up until this point in our proof. 

Similarly, $X_{13}$ and $X_{15}$ must be in $\bs X$, because otherwise $\mc N_{\bs X} (X_4)$ and $\mc N_{\bs X} (X_6)$ would be indentable. That $X_8\notin \bs X$ and $X_{10}\notin \bs X$ follows from the observations that inclusion of either of them would create a double deck with the bottom row composed of $X_7,X_1,X_2, X_3$ or $X_1,X_2, X_3,X_{11}$, and therefore would make $\bs X$ 2-indentable due to Lemma~\ref{ddeckcollpos}. After all these steps, we obtain a Human shown in Figure~\ref{configaenum}-iv). 
 Proof goes in the same manner if we initially assume $\bs X$ includes $\bs{R}_6^\top$ rather than $\bs R_6$, but we get a rotation symmetry of Human. 
\end{proof}

We next give a basic observation that will be useful in the coming lemmas.
\begin{lemma} \label{lsubsetx}
    Given a loop $\bs L$ in a clump $\bs X$, its inland satisfies
    \[\inl(\bs L)\subseteq |\bs X|.\]
\end{lemma}
\begin{proof}
\rev{
Assume, on the contrary, there exists $y\in \mathbb{R}^2$ such that $y\in \inl(\bs L)$ and $y\notin |\bs X|$. As $|\bs L|\subseteq |\bs X|$, $y$ has to be in the interior of the scaffold of $\bs L$ which is a simple closed loop in $\mathbb{R}^2$. Take the union, $C$, of all connected subsets of $\mathbb{R}^2\setminus |\bs X|$ that contain $y$. $C$ is bounded as $y$ is in the interior of $\scf(\bs L)$ due to Jordan curve theorem. That $\mathbb{R}^2\setminus|\bs X|$ has a bounded connected component $C$ contradicts with the contractibility of the compact, locally contractible set $|\bs X|$ as a CW complex, due to Alexander duality.  
}
\end{proof}

\begin{lemma}\label{otherHumanlem}
    Assume a clump $\bs X$ has a Human, cubes of which are labeled as in Figure~\ref{configaenum}-iv. Define the cubical set $\bs Z$ as the set of all cubes of $\bs X$ reachable through a path in $\bs X$ from $X_{17}$ that does not go through $X_2$, i.e.,
    \begin{equation*}
   \begin{split}
       \bs Z\coloneqq \{Z\in \bs X, \text{ $\exists$  a path $\bs P= \{P_1, \ldots, P_n\}\subset \bs X$ \text{ for some $n\geq 1$}}, \\\text{ s.t. $P_1 = X_{17}, P_n = Z, X_2\notin \bs P$} \}.
    \end{split}
   \end{equation*}

   Then $\bs Z$ is a clump.
\end{lemma}

\begin{proof}
Note that there cannot be a cube in $\bs Z$ that share a 1-dimensional face (edge) with any of the cubes of the Human, because this would mean that there is a loop in $\bs X$ starting at $X_9$ and ending at $X_2$, and the inland of this loop would include at least one of the empty locations (denoted with crosses in Figure~\ref{configaenum}-iv), which would contradict with Lemma~\ref{lsubsetx}.
Furthermore, \rev{cubes of $\bs Z$ cannot share a 0-dimensional face (point) with the cubes of the human} since this would make the intersection an irregular point (see Figure~\ref{irregpiclabel} left hand side), contradicting that $\bs X$ is a clump. Therefore $|\bs Z|\cap |X_i|= \varnothing$ for all $1\leq i\leq 16$, and $\mc N_{\bs Z}(Z)=N_{\bs X}(Z)$ for all $Z\in\bs Z$, which together with Lemma~\ref{regiffneighlem} and that $\bs X$ is a clump proves that $\bs Z$ is regular. 

We are next going to prove that $|\bs Z|$ is contractible. Assume that it is not contractible. Then, since $|Z|$ is a planar cubical complex it means either $|\bs Z|$ is not path-connected or it has a non-trivial 1-dimensional homology. $|\bs Z|$ has to be path-connected by definition, so has to have a a non-trivial 1-dimensional homology. Given this, take a simple closed loop in $|\bs Z|$ that has a hole, which can be turned into a loop $\bs L$ of cubes (Definition~\ref{pathdef}) in $\bs Z$ since $\bs Z$ does not have irregular points. Due to Lemma~\ref{lsubsetx} inland of this loop has to have a cube, call it $C$, in $\bs X$ but not in $\bs Z$. Note that $C$ is connected by a path in $\bs X$ not going through $X_2$ to a cube $L\in\bs L$. Being in $\bs Z$, by definition, $L$ is connected by a path to $X_{17}$ not going through $X_2$. Concatenating these two paths (and removing any loops if occur because of the concatenation) gives a path from $X_{17}$ to $C$ not going through $X_2$, and therefore $C\in \bs Z$, which is a contradiction.   
This finishes the proof. 
\end{proof}

\begin{corollary}\label{trum2colcor}
    If a clump, $\bs X$, has a Human then it is 2-indentable.
\end{corollary}
\begin{proof}
  Using Lemma~\ref{otherHumanlem}, we argue that $\bs X$ has a clump that does not intersect the Human. On this clump we use Lemma~\ref{lem ab} and Lemma~\ref{norectlem} to deduce that it is 2-indentable or has a Human of itself. This implies that $\bs X$ is either 2-indentable or it has two disjoint Humans. Noting that a Human is always 1-indentable
  concludes the proof.    
\end{proof}

Our next lemma will take care of the ``branchy'' clumps that do not have a 2-by-3 rectangle.  
\begin{lemma}\label{norectlem}
    If a clump does not include a 2-by-3 rectangle then it is 2-indentable.
\end{lemma}
\begin{proof}
  We will first show that if a clump $\bs X$ does not have a 2-by-3 rectangle then then it does not have a loop of length at least 6. For this, assume $\bs X$ has a loop of length 6 or longer. 
  We take one such loop of $\bs X$ and consider its scaffold. Note that this is a self avoiding polygon enclosing an area of at least 2 and therefore its interior includes the skeleton \rev{(Definition~\ref{defSkel})} of an $\bs R_6$, which is composed of two squares sharing an edge. Using Lemma~\ref{lsubsetx}, we conclude $\bs X$ has a 2-by-3 rectangle.

Now assume $\bs X$ does not have a loop of length 6 longer. We take the longest path (choose an arbitrary one if there are multiple) in $\bs X$ that we name $\bs P = \{P_1, \ldots, P_n\}$. Without loss of generality, due to symmetry, assume that the community of $P_n$ looks as in Figure~\ref{pathlongestfig} below. 

 \begin{figure}[ht]
\centering
\begin{tikzpicture}[scale=1, line join=round, line cap=round]

\foreach \x in {0,2}
\foreach \y in {0,1,2}
{
        \draw[very thick, fill=gray!10,dashed] (\x,\y) rectangle ++(1,1);
}

\draw[very thick, fill=gray!10,dashed] (1,2) rectangle ++(1,1);

\foreach \y in {0,1} 
{\draw[very thick, fill=gray!40] (1,\y) rectangle ++(1,1);}

\node[anchor=center] (Label) at (
        1.5, 1.5) {$P_n$};
\node[anchor=center] (Label) at (
        1.5, 0.5) {$P_{n-1}$};

\node[anchor=center] (Label) at (
        0.5, 2.5) {$X_1$};

\node[anchor=center] (Label) at (
        1.5, 2.5) {$X_2$};

\node[anchor=center] (Label) at (
        2.5, 2.5) {$X_3$};

\node[anchor=center] (Label) at (
        2.5, 1.5) {$X_4$};
\node[anchor=center] (Label) at (
        2.5, 0.5) {$X_5$};
\node[anchor=center] (Label) at (
        0.5, 0.5) {$X_6$};
\node[anchor=center] (Label) at (
        0.5, 1.5) {$X_7$};

\end{tikzpicture}
\caption{Community of the last cube in the longest path in $\bs X$.}
\label{pathlongestfig}
\end{figure}

\rev{First, note that for each $i\in\{2,4,7\}$ either $X_i\in\bs X$, $X_i\in \bs P$ or $X_i\notin\bs X$, $X_i\notin \bs P$ because in the case $X_i\in\bs X$, $X_i\notin \bs P$ the path $\{P_1, \ldots, P_n, X_i\}$ would be a path longer than $\bs P$. Moreover, $X_2$ cannot be on $\bs P$, since otherwise if there exists $1\leq j\leq n-2$ such that $X_2=P_j$ then $\{P_j,\ldots,P_{n}\}$ would constitute a loop of length six or greater. Therefore we conclude that  $X_2\notin\bs P$, $X_2\not\in \bs X$.  Assume now that $X_7\in \bs P$, i.e., there exists $1\leq j\leq n-2$ such that $X_7=P_j$. If $X_6\notin \bs P$ then $ \{P_j, \ldots, P_{n-1}, P_n\}$ would be a loop of length at least 6. We conclude that if $X_7\in \bs X$ then $X_6\in \bs X$. Similarly, if $X_4$ is in $\bs X$ then $X_5$ is too. It is straightforward to check that $\mathcal{N}_{\bs X} (P_n)$ is indentable in all possible cases, i.e., $X_6,X_7\in\bs X$; $X_7\notin\bs X$; $X_5,X_4\in\bs X$; $X_4\notin\bs X$. Repeating the argument for $P_1$ instead of $P_n$ results in that $\mathcal{N}_{\bs X} (P_1)$ is also indentable. Using Lemma~\ref{neighbcollem}, we obtain that $P_1$ and $P_n$ are in the indent set of $\bs X$. This concludes the proof.} 
\end{proof}

\begin{proof}[Proof of Theorem~\ref{indentthm}]
  Follows from Lemma~\ref{norectlem}, Lemma~\ref{lem ab}, and \allowbreak Corollary~\ref{trum2colcor}.
\end{proof}

Using  Theorem~\ref{indentthm}, we will be able to prove the ergodicity of our dynamic cubical set model, to be discussed in the next section, for the special case of $d=2$.

\section{Ising Ensemble on Clumps}
\label{Isingsec}
In this section, we construct a Markov process that samples random clumps in any dimension and that models a randomly evolving Markovian cubical set. The state space of the process is the set of cubical sets in $\Rd$, i.e., finite subsets of $\mc{K}^d$. 
The process starts with the assignment $\bs X(t) =\left\{ [-\frac{1}{2},\frac{1}{2}]^d\right\}$ for $t=0$, and an independent unit-rate Poisson counting process with an exponential clock $C_K:\rpl\to \{0,1,2,\ldots\}$ associated with each $K\in \mc{K}^d$. Define the set   
\begin{align*}
        \mc M (\bs X(t))\coloneqq \left\{\bigcup_{X\in \bs X(t)} \mc N_{\bs X(t)} (X)\right\} \setminus \left[-\frac{1}{2},\frac{1}{2}\right]^d,
    \end{align*}
and the $i\ts{th}$ transition time, $i\in\{1,2,\ldots\}$,
\begin{align*}
    \tau_i \coloneqq \inf\{t>\tau_{i-1}: \text{$C_X(t)> C_X(\tau_{i-1})$ for some $X\in \mc M (\bs X(\tau_{i-1}))$ } \}
\end{align*}
with $\tau_0=0$, and 
\begin{align*}
    Y_i = \{\text{$X\in \mc M (\bs X(\tau_{i-1})):$ $C_X(\tau_i)> C_X(\tau_{i-1})$ } \}.
\end{align*}

In words, $\tau_i$ denotes the first tick in the independent exponential clocks of the set of all cubes that are in $\bs X(\tau_{i-1})$ and their neighbors, and $Y_i$ denotes the cube whose clock made the first tick.

For the interval between ticks, i.e. for all $\tau_{i-1}\leq t<\tau_i$, we assume $\bs X(t) = \bs X(\tau_{i-1})$.
The transition at time $\tau_i$ is decided as follows. If $\mathcal{N}_{\bs X(\tau_{i-1})}(Y_i)$ is not an indentable clump then $\bs X(\tau_i)=\bs X (\tau_{i-1})$. Otherwise, if $Y_i\in \bs X(\tau_{i-1})$ then
    \begin{align*}
        \bs X(\tau_i) = \begin{cases}
            \bs X(\tau_{i-1})\setminus Y_i &\textit{w.p.\quad $p_\downarrow$},\\
            \bs X(\tau_{i-1}) &\textit{w.p.\quad $1-p_\downarrow$,}
        \end{cases}
    \end{align*}
    and if $Y_i\not\in \bs X(\tau_{i-1})$ then
    \begin{align*}
        \bs X(\tau_i) = \begin{cases}
            \bs X(\tau_{i-1})\cup Y_i &\textit{w.p.\quad $p_\uparrow$},\\
            \bs X(\tau_{i-1}) &\textit{w.p.\quad $1-p_\uparrow$.}
        \end{cases}
    \end{align*}
    
     Here $p_\uparrow$ and $p_\downarrow$ refer to \emph{expand} and \emph{indent probabilities} respectively, for some $0< p_\downarrow,p_\uparrow\leq 1$. The ratio, $p_\uparrow/p_\downarrow$, of these probabilities, which determines the tendency of the clump to grow, will play an important role in our analysis of the CTMC in the planar special case. An analogous parameter is called the (area) \emph{fugacity} in the study of self-avoiding polygons through generating functions \cite{guttbook}. The relationship will be made more clear in Remark~\ref{fugabetarmk}.

\begin{remark}
    Due to Lemma~\ref{neighbcollem} and Corollary~\ref{neighbexplem}, under the rules of Markov chain given above, $\bs X(t)$ is a clump for all $t\geq 0$.
\end{remark}

\subsection{Markov model for \texorpdfstring{$d=2$}{}}
Under this subsection, we will study the CTMC defined above for general $d\geq 2$, in the planar special case $d=2$. The planar assumption  will be assumed to hold throughout even when it is not explicitly mentioned. 

We will first relate the ergodic behavior of the Markov process to the combinatorial properties of \emph{self-avoiding polygons}, defined below.
\begin{definition}
A \emph{self-avoiding circuit} that starts at the origin is a sequence of points $Z_1,Z_2,\cdots,Z_n$ in $\mathbb{Z}^2$ such that $Z_1=Z_n=0$, and for all $1\leq i,j\leq n-1$, $|Z_i-Z_{i+1}|_2=1$ and $Z_i=Z_j$ only if $i=j$. A \emph{self-avoiding polygon} is the interior (in $\bb R^2$) of a self avoiding circuit.  
\end{definition}

Next, we give some remarks below that will enable us to study the statistical properties of our Markov model.
\begin{remark}\label{bijsaprmk}
    \rev{Note that due to Theorem~\ref{indentthm}, the ``reachable state space'' (\emph{support}) of the CTMC in the special case of $d=2$ is the set of all clumps whose geometric realization includes the origin. Furthermore, for all $n\geq 1$, 
    there is a bijection between the set of all clumps of size $n$ which include the origin and the set of all self-avoiding polygons of area $n$. This follows from observing that the boundary of the geometric realization of every clump covering the origin defines a self-avoiding polygon (after translation) and vice versa. Moreover, there is an $1$-to-$n$ surjective mapping between the set of self-avoiding polygons and the set of   
     \emph{unrooted} self-avoiding polygons, i.e., the set of self-avoiding polygons up to translation, an object of significant interest in the statistical physics literature \cite{guttbook}. Therefore, if $a_n$ is defined to be the number of unrooted self-avoiding polygons of area $n$, the cardinality of the set of clumps of size $n$ which include the origin is upper bounded by $na_n$.} 
\end{remark}


\begin{remark}\label{saprmk}
\rev{Enumerating the unrooted self-avoiding polygons is an active research area. The classical concatenation arguments (i.e. that every self-avoiding polygon can be decomposed into two separate ones, but not vice versa, see Chapter-1 of \cite{guttbook} for details) gives the inequality $a_{n+m}\leq a_na_m$, which in turn implies the sub-additivity of $\log a_n$. Using Fekete's lemma,} 
\begin{equation}\label{anlimeqn}
   \lim_{n\to\infty} (a_n)^{1/n} = \kappa 
\end{equation} 
for some $\kappa\in (0,\infty)$. This implies for any constant $c_1>1$ there exists a constant $C>0$ such that
\begin{align}\label{anbound}
   a_n\leq C \cdot(c_1\kappa)^n. 
\end{align} 
\end{remark}

Now we are ready to state the main theorem of this section. 
\begin{theorem}\label{stationarythm}
    The continuous-time Markov chain $\bs X(t)$, $t\in\rpl$, described previously is irreducible. 
     Furthermore, if $\beta\coloneqq \frac{p_\uparrow}{p_\downarrow}<\frac{1}{\kappa}$, \rev{$\kappa$ defined in Remark~\ref{saprmk}}, then $\bs X(t)$ is also ergodic and has the stationary distribution $\Pi$, given as (up to normalization) 
 \begin{align}
     \Pi (\bs Y) \propto \begin{cases}
        \beta^{\card(\bs Y)} \quad &\text{if $\bs Y$ is a clump and $[-1/2,1/2]^2\in \bs 
        Y$,}\\
        0 &\text{otherwise.}
    \end{cases}\label{pidefeq}
    \end{align}
\end{theorem}
Before the proof of the theorem above, we give an auxiliary lemma as follows.
\begin{lemma}\label{pigoodlem}
    If $\beta<\frac{1}{\kappa}$, $\Pi(\bs Y)$ given in \eqref{pidefeq} is a well-defined distribution on cubical sets. Furthermore, all finite moments of $\card(\bs Y)$ under $\bs Y\sim\Pi$ are finite.  
\end{lemma}
\begin{proof}
    First, we prove that $\Pi$ given in \eqref{pidefeq} is a valid distribution on the space of cubical sets on $\mathbb{R}^2$. Letting $\tilde \Pi$ denote the unnormalized $\Pi$, we note that
    \begin{align*}
        \sum_{\bs Y\subset \mathcal{K}^2} \tilde\Pi(\bs Y) =  \sum_{n=1}^{\infty}\sum_{\stackrel{\bs Y\subset \mathcal{K}^2}{\card(\bs Y)=n}}\beta^{n}\cdot \ind\{\text{$\bs Y$ is a clump and $[-1/2,1/2]^2\in \bs 
        Y$}\}
    \end{align*}
    where the sum is over all cubical sets on $\mathbb R^2$.

    Remark~\ref{bijsaprmk} gives  
    \begin{align}
        \sum_{\stackrel{\bs Y\subset \mathcal{K}^2}{\card(\bs Y)=n}} \ind\{\text{$\bs Y$ is a clump and $[-1/2,1/2]^2\in \bs  
        Y$}\} \leq n a_n,\label{sumindyclump}
    \end{align}
 where $a_n$ was defined as the number of unrooted self-avoiding polygons of area $n$. Combining \eqref{sumindyclump} with the previous equality and using \eqref{anbound}, we obtain that for any $c_1>1$
\begin{align}\label{sumpitileqn}
        \sum_{\bs Y\subset \mathcal{K}^2} \tilde\Pi(\bs Y) &= \sum_{n=1}^{\infty}\beta^{n}n a_n
        \leq C\cdot \sum_{n=1}^{\infty} (c_1 \kappa)^n \beta^{n}n, 
\end{align}
    for some $C>0$ \eqref{anbound}. Choosing $c_1 = \sqrt{1/\kappa\beta }$,
    \begin{align*}
        \sum_{\bs Y\subset \mathcal{K}^2} \tilde \Pi(\bs Y) \leq C\cdot \sum_{n=1}^{\infty} (\sqrt{\kappa\beta})^n n <\infty 
    \end{align*}
    as $\kappa\beta<1$, and therefore, $\tilde \Pi$ is normalizable giving the distribution $\Pi$.
    
    To prove that $\card(\bs Y)$ has finite moments under $\Pi$, we write, for any $c_1>1$
    \begin{align*}
        \ex_{\Pi}[(\card(\bs Y))^k] \leq  C'\cdot \sum_{n=1}^{\infty} (c_1 \kappa)^n \beta^{n}n^{k+1}, 
    \end{align*} 
for some constant $C'>0$, which can be chosen as the ratio of $C$ \eqref{anbound} and the normalization constant of $\tilde\Pi$. Choosing again $c_1 =  \sqrt{1/\kappa\beta }$, we obtain
\begin{align*}
        \ex_{\Pi}[(\card(\bs Y))^k]   \leq C'\cdot \sum_{n=1}^{\infty} (\sqrt{\kappa\beta})^n n^{k+1} <\infty. 
    \end{align*}
for all $k\in\{1,2,\ldots\}$.
\end{proof}

Now we are ready to prove Theorem~\ref{stationarythm}. 
\begin{proof}[Proof of Theorem~\ref{stationarythm}]
    Note the embedded discrete-time Markov chain
    \begin{align*}
        \hat{\bs X}_i = \bs X(\tau_i),
    \end{align*}    
    $i\in \zpl$ of $\bs X(t)$ is irreducible due to Theorem~\ref{indentthm}. This is because any pair of clumps $\bs V, \bs W$ is connected through the indent of $\bs V$ to the single cube at the origin followed by reversing the indent steps of $\bs W$ as expansions. This proves the irreducibility of $\bs X(t)$. Note that 2-indentability is crucial for this argument since the CTMC never chooses the cube at the origin to indent, therefore there always needs to be another cube with an indentable community.

    Next, we show that $\Pi$ satisfies the detailed balance equation for the CTMC 
    \begin{align}\label{baleqn}
        \lambda_{\bs Y, \bs Z}(t) {\Pi(\bs Y)}=\lambda_{\bs Z, \bs Y}(t) {\Pi(\bs Z)} 
    \end{align}
for all $t\geq 0$, and all clumps $\bs Y, \bs Z$ that contain $[-1/2,1/2]^2$, where the transition rates $\lambda_{\bs Y, \bs Z}$ at time $t$ is defined as
\[
\lambda_{\bs Y, \bs Z}(t) \coloneqq \lim_{h\to 0} \frac{\pr\left[\bs X(t+h) = \bs Z | \bs X(t)=\bs Y \right]} {h}.
\]
First note that
if $\bs Y = \bs Z$ then \eqref{baleqn} is trivially satisfied. Otherwise, if $\card(\bs Y) = \card(\bs Z) + 1$, and $\bs Z$ is indentable to $\bs Y$, then $\Pi(\bs Y) = \beta\cdot \Pi(\bs Z)$. Also note that
\[ \pr\left[\bs X(t+h) = \bs Y | \bs X(t)=\bs Z \right] = p_\uparrow(1-e^{- h}) e^{-h  (\card(\mc M(\bs Z))-1)} +o(h),\]
\rev{where the $1-e^{-h}$ term is the probability of the clock of $\bs Z\setminus \bs Y$ ticking between time $t$ and $t+h$, $e^{-h  (\card(\mc M(\bs Z))-1)}$ is the probability of the clocks of candidate cubes \emph{not} ticking between $t$ and $t+h$, and $o(h)$ term accounts for more than one evenr happening in the time interval.} Similarly,
\[ \pr\left[\bs X(t+h) = \bs Z | \bs X(t)=\bs Y \right] = p_\downarrow(1-e^{- h}) e^{-h (\card(\mc M(\bs Y))-1)} +o(h) .\]
Therefore,
\[ \frac{\lambda_{\bs Z,\bs Y}(t)}{\lambda_{\bs Y,\bs Z}(t)} = \lim_{h\to 0}\frac{ \pr\left[\bs X(t+h) = \bs Y | \bs X(t)=\bs Z \right]} { \pr\left[\bs X(t+h) = \bs Z | \bs X(t)=\bs Y \right]} =  \frac{p_\uparrow}{p_\downarrow} = \beta,  \]
establishing \eqref{baleqn} for this case. On the other hand, if $\card(\bs Z) = \card(\bs Y) + 1$ and $\bs Y$ is indentable to $\bs Z$, 
\begin{equation*}
\begin{split}
\frac{\lambda_{\bs Z,\bs Y}(t)}{\lambda_{\bs Y,\bs Z}(t)} = \lim_{h\to 0}\frac{ \pr\left[\bs X(t+h) = \bs Y | \bs X(t)=\bs Z \right]} { \pr\left[\bs X(t+h) = \bs Z | \bs X(t)=\bs Y \right]} =  \frac{p_\downarrow}{p_\uparrow} = &\frac{1}{\beta} \\
=& \frac{\beta^{\card(\bs Y)}}{\beta^{\card(\bs Z)}}=\frac{\Pi(\bs Y)}{\Pi(\bs Z)}.
\end{split}
\end{equation*}

For all other cases, both sides of \eqref{baleqn} is zero, establishing the detailed balance equation, and therefore that $\Pi$ is the stationary distribution.

To conclude ergodicity we lastly need to show that the CTMC is \emph{nonexplosive}, that is, the probability that it makes infinitely many transitions in finite time is zero. This follows from the criteria that prohibits explosions given in the Corollary 4.4 of \cite{asmussenbook}, which translates to our setting as
\begin{align*}
     \sum_{\bs Y\subset \mathcal{K}^2} \Pi(\bs Y)\card(\mc M(\bs Y))<\infty.
\end{align*}
Note that for all $\bs Y\subset \mathcal{K}^2$, \[\card(\mc M(\bs Y))\leq 9\card(\bs Y), \]
and therefore
\begin{align*}
     \sum_{\bs Y\subset \mathcal{K}^2}  \Pi(\bs Y)\card(\mc M(\bs Y))\leq 9\sum_{\bs Y\subset \mathcal{K}^2} \Pi(\bs Y)\card(\bs Y) &= 9\cdot\ex_\Pi[\card(\bs Y)]\\&<\infty,
\end{align*}
due to Lemma~\ref{pigoodlem}.
\end{proof}

\begin{corollary}
If $\beta= \frac{p_\uparrow}{p_\downarrow}<\frac{1}{\kappa}$, for all cubical sets $\bs Y$, the CTMC, $\bs X(t)$, satisfies  
\[\lim_{t\to\infty} \pr[\bs 
 X(t) =\bs Y ] = \Pi(\bs Y).\]
 \end{corollary}
 \begin{remark}\label{fugabetarmk}
     As would be expected based on Theorem~\ref{stationarythm}, numerical experiments we carried out (see Figure~\ref{snapshots}) show an excellent alignment of the critical $\beta$ for the existence of a stationary distribution and the inverse of the critical fugacity ($\kappa$) available in the self-avoiding polygon literature. We have observed in our experiments that for fixed $p_{\downarrow}=1$, the clump in the CTMC is observed to stay bounded when $p_{\uparrow}<0.25$ and grow indefinitely when $p_{\uparrow}>0.25$. The numerically estimated values for the growth constant for the square lattice $\kappa$, see p.6 of \cite{guttbook}, suggests the phase change to be at around $p_{\uparrow}=0.252\approx 1/3.971\approx 1/\kappa$. In the over-critical regime, the value of $p_{\uparrow}$ seems to determine the ``shape'' of the growing clump, the lower $p_{\uparrow}$ values correspond to the ``branchier'' clumps.
 \end{remark}

\begin{figure}[ht]
\centering
\includegraphics[width=.95\textwidth]{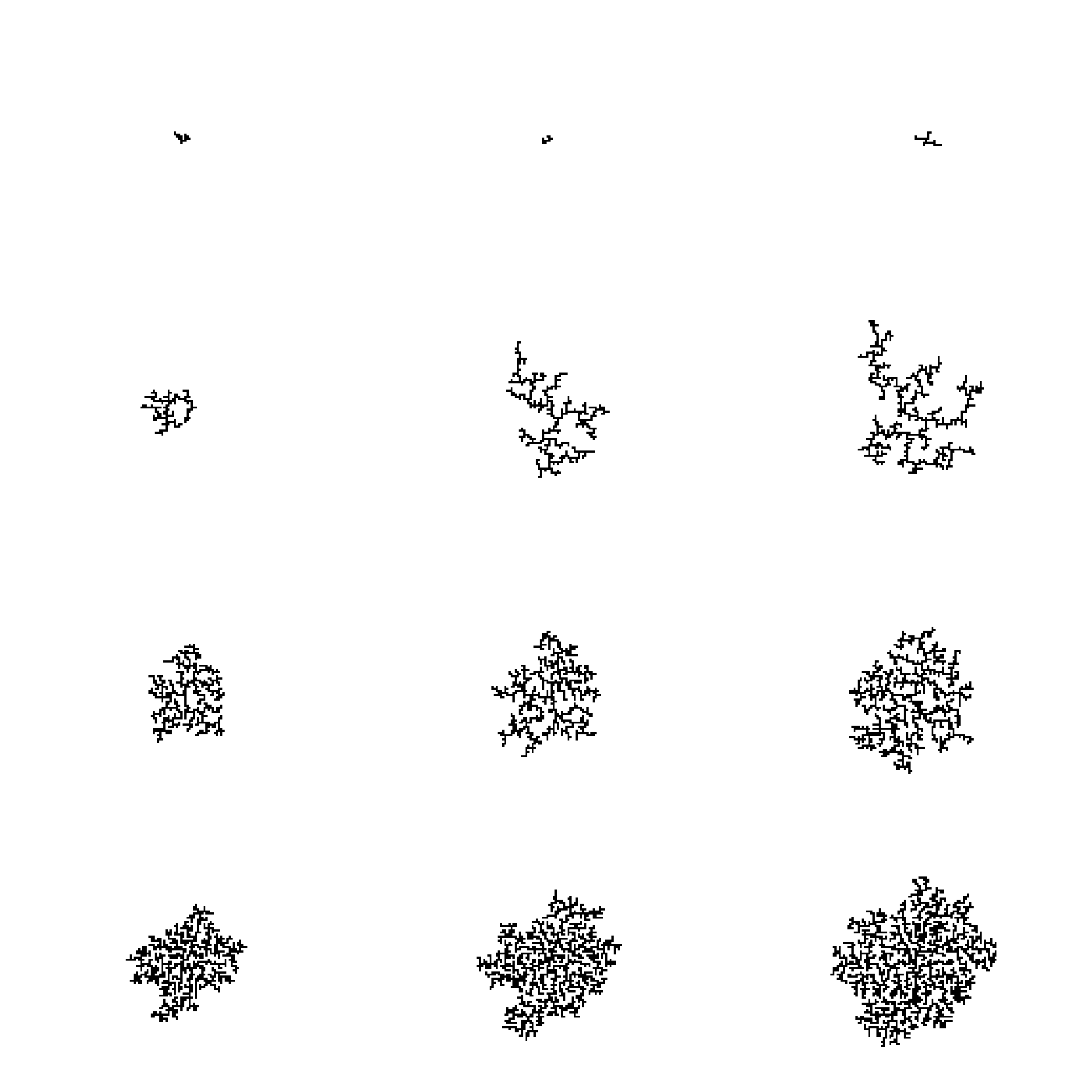}
    \caption{Snapshots of the Markov dynamics for fixed $p_{\downarrow}=1$ and different $p_{\uparrow}$ values. The rows include regularly sampled snapshots for $p_{\uparrow}$ values 0.24, 0.26, 0.4, 0.8 in order. Sampling frequency is adjusted across different $p_{\uparrow}$ to better illustrate the evolution of the cubical set. Note the phase change between $p_{\uparrow}=0.24$, and $p_{\uparrow}=0.26$, as predicted in \cite{guttbook}.} 
    \label{snapshots}
\end{figure}

The Theorem~\ref{stationarythm} is true beyond $d=2$ case but we postpone the proof till a later publication, and state for now the following.
\begin{conjecture}
    Theorem~\ref{stationarythm} holds for all $d\geq 2$ with $d$-dependent constants $\kappa(d)$.
\end{conjecture}

\section{\texorpdfstring{\rev{Rare Events}}{}}

In this part, we will present a Poisson process approximation theorem for the ``rare'' time instances that the clump in CTMC reaches ``unusually'' large sizes. \rev{The premise here is that the maximum size of the clump for each time interval between consecutive visits of the Markov process to the single cube is independent, and with the correct scaling of time and the choice of threshold size, the limiting distribution of the number of occurrences of unusually large clumps converge in distribution to Poisson.} However, given that the dynamic clump passes a threshold of a large size at some time point, one would expect that it will have fluctuations of going up and down with positive probability, therefore a Poisson process approximation would not hold in the strict sense. However, a point process that counts each ``cluster'' of fluctuations as one, might qualify for Poisson approximation with correct time scaling. It turns out this is correct. We first introduce some notation in order to define how we count clusters. We define $\zeta_i$ to be the time that $\bs X(t)$ comes back to being composed of only the unit square in the origin $i$\ts{th} time, and we let $\xi_i$ be the largest size the clump has reached in the  $i$\ts{th} interval between successive returns to the origin. Formal definitions follow.
\begin{definition}
Assign $\zeta_0=0$. Given $\bs X(t)$, for $i\in\zpl$, define 
\begin{align*}
    \zeta_{i+1} \coloneqq \inf\{\text{$t:$ $t>\zeta_{i}$, $\card(\bs X(t))=1$, $\exists\, \hat{t}$ with}&\text{ $\zeta_{i}<\hat t<t$,}\\&\text{  and $\card(\bs X(\hat t))=2$}  \},
    \end{align*}
    and 
    \begin{align}
    \xi_{i+1} &\coloneqq \sup_{\zeta_{i}\leq t<\zeta_{i+1} } \card(\bs X(t)).\label{xii_def_eq}
\end{align}
\end{definition} 

\begin{definition}\label{mudef}
If $\beta<\frac{1}{\kappa}$ then ergodicity and the existence of the stationary distribution $\Pi$ imply the positive recurrence of the CTMC (see Theorem 3.5.3 of \cite{norrisbook}), which leads us to define the (finite) expected recurrence time, 
\begin{align*}
    \mu\coloneqq \ex[\zeta_1]\in(0,\infty).
\end{align*}
\end{definition}

Furthermore, we will define below $\zeta^T_i$ to be the $i$\ts{th} time that $\bs X(t)$ reaches the size $T>0$ after being of size one. The last time $\bs X(t)$ was of size one before $\zeta^T_i$ will be denoted with $\zeta_{\sigma(i)}$, and the first time it reaches size $T$ after $\zeta_i$ will be denoted with $\zeta^T_{\upsilon(i)}$.
\begin{definition}\label{timelinedefs}
Assume $\beta<\frac{1}{\kappa}$, and let, for all $T>0$,
\begin{align}\label{etaTdefeqn}
     \eta_T \coloneqq \frac{ \mu  }{\pr[\xi_1\geq T]}.
\end{align} 
 Given $\bs X(t)$ and $T$, assign $\zeta^T_0=0$, and for $i\in\zpl$, define 
\begin{align*}
\begin{split}
    \zeta^T_{i+1} \coloneqq \inf\{\text{$t:$ $t>\zeta^T_{i}$, $\card(\bs X(t))\geq T$, $\exists\, \hat{t}$}&\text{ with $\zeta^T_{i}<\hat t<t$}\\ &\text{ and $\card(\bs X(\hat t))=1$}\}.
\end{split}
\end{align*}
Lastly, define $\sigma(i):\zpl\setminus\{0\}\to\zpl$ as
\begin{align*}
    \sigma(i)\coloneqq \sup\{k:\, k\geq 0,\, \zeta_k< \zeta^T_i \},
\end{align*}
and $\upsilon(i):\zpl\to\zpl\setminus\{0\}$,
\begin{align*}
    \upsilon(i)\coloneqq \inf\{k:\, k\geq 1,\, \zeta_i< \zeta^T_k \}.
\end{align*}
\end{definition}
 
\begin{remark}
    By definition, for all $i\in\zpl$,
    \begin{align*}
        \zeta_i\leq \zeta_{\sigma(\upsilon(i))}<\zeta^T_{\upsilon(i)} <\zeta_{\sigma(\upsilon(i))+1}.
    \end{align*}
\end{remark}

 \begin{figure}[ht]
\centering
\begin{tikzpicture}
    \draw[thick] (0,0) -- (4.1,0);
    \draw[thick] (4.8,0) -- (6.9,0);
    \draw[thick,->] (7.6,0) -- (9.5,0);

    \node at (4.45,0) {$\cdots$};
    \node at (7.3,0) {$\cdots$};
    \node[above] at (-2,0.12) {$t=$};
    \node[below] at (-2,-0.12) {$\card(\bs X(t))=$};

    \foreach \x/\t/\event in { 0/{$\zeta_0=\zeta^T_0=0$}/{$1$}, 
    1/{}/{$2$},
    1.5/{}/{$3$},
    1.8/{}/{$2$},
    2.4/{$\zeta_1$}/{$1$},
    2.8/{}/{$2$},
    3.5/{$\zeta_2$}/{$1$},
    3.9/{}/{$2$},
    5.1/{$\zeta^T_1$}/{$T$},
    6.0/{}/{$T+1$},
    6.8/{}/{$T$},
    8.0/{}/{$2$},
    8.4/{$\zeta_3$}/{$1$}
    }
        {
        \draw[very thick] (\x,0.12) -- (\x,-0.12);
        \node[above] at (\x,0.12) {\t};
        \node[below] at (\x,-0.12) {\event};
    } 
\end{tikzpicture}
\caption{An example timeline illustrating the events defined in Definition~\ref{timelinedefs}. The ticks represent the time instances where the size of the clump changes. Below the timeline are the sizes of the clump at those times. Specific time instances that we have named are shown above.  Note $\sigma(1) = 2$, $\upsilon(0)=\upsilon(1) =1$, and $\xi_1=3$, $\xi_2=2$, $\xi_3\geq T+1$ for this example.}
\label{timelinefig}
\end{figure}

The asymptotics of the quantity $\eta_T$ will control the expected \emph{sojourn time} of the process from a single square to a clump of size $L$. 

Lastly, we give necessary formal definitions for point processes and their convergence in distribution. More abstract and general theory of point processes can be found in \cite{kallenbergbook}. The following special case, though, will suffice for our purposes. 
\begin{definition}\label{ppdef}
    A (simple) point process $N$ on $\rpl$ is a random counting measure in the form 
    \begin{equation*}
        N = \sum_{i}\delta_{\tau_i},
    \end{equation*}
    where $\delta_x$ denotes the Dirac measure at $x\in \rpl$, and $\tau_i\in \rpl$ are distinct for $i\in\{1,2,\ldots\}$.
    Given a point process $N$, and a sequence of point processes $N_1, N_2,\ldots$, it is said that $N_n$ \emph{converges in distribution (under the vague topology)} to $N$ if and only if the sequence of random variables
    $\int f \der N_n$    
    converge in distribution to $\int f \der N$ for every bounded measurable continuous function $f$ with bounded support on non-negative reals. We denote convergence in distribution also as 
    \[N_n \xrightarrow[n\to\infty]{\mathtt{d}} N.\]
\end{definition}
 
 The following is a standard result in the study of point processes.
 \begin{theorem} [Theorem A1 in \cite{leadbetterbook}]\label{leadbthm}
     Given the notation in Definition~\ref{ppdef}, $N_n \xrightarrow[n\to\infty]{\mathtt{d}} N$ if 
\begin{align*}
    \lim_{n\to\infty}\ex[N_n((c,d])] =
    \ex[N((c,d])]
\end{align*}
     for all $0<t_1<t_2<\infty$, and
\begin{align}\label{vacumprob}
    \lim_{n\to\infty}\pr[N_n(B)=0] = \pr[N(B)=0], 
\end{align}
     for all $B=\cup_{i=1}^k(t_{1,i},t_{2,i}]$ where $0<t_{1,i}<t_{2,i}<\infty$, $1\leq i\leq k$.
 \end{theorem}

 Equipped with the necessary terminology, now we are ready to state our main theorem regarding Poisson convergence.

\begin{theorem}  \label{convpoissonthm}
Define the counting process
\begin{align}\nn
     N'_T \coloneqq \sum_{i\geq 1}  \delta_{\zeta^T_i/\eta_T}.
 \end{align}

 If $\beta<\frac{1}{\kappa}$ then
 \begin{align*}
      N'_T \xrightarrow[T\to\infty]{\mathtt{d}} N^*,
 \end{align*} 
  where with $N^*$, we denote the homogeneous Poisson point process on the nonnegative real line with unit intensity.
\end{theorem}

Note that $N'_T$, defined above, counts the number of \rev{occurrences that} $\bs X(t)$ is of size $T$ after indenting to a single square. We will prove Theorem~\ref{convpoissonthm} by showing the Poisson convergence of a ``related'' discrete-time process of exceedances. We start with some auxiliary lemmas below. The first lemma formalizes what we mean by a ``related'' process and how this notion connects to the convergence in distribution. The second one will introduce and prove the convergence of the discrete-time process that is of interest to us. 

\begin{lemma}\label{iffconvdlem}
    Given two sequences of point processes $N_k$ and $\hat{N}_k$, $k\in\{1,2,\ldots\}$ on $\rpl$,  satisfying
    \begin{equation*}
        \lim_{k\to\infty} \pr[N_k(I) = \hat{N}_k(I)]=1,
    \end{equation*}
    for any semi-closed interval $I=[t_1,t_2)\subset\rpl$. Then, for any point process $N$ on $\rpl$
    \begin{equation*}
        N_k \xrightarrow[k\to\infty]{\mathtt{d}} N \iff \hat{N}_k \xrightarrow[k\to\infty]{\mathtt{d}} N.
    \end{equation*}
\end{lemma}
\begin{proof}
    See Lemma 3.3 of \cite{movingavg}, proof of which is, in turn, based on Theorem 3.1.7 of \cite{kerstaninfdiv}.
\end{proof}
\begin{lemma}\label{disctimelem}
 Given the CTMC, $\bs X(t)$, define the exceedance point process of $\xi_i$'s \eqref{xii_def_eq} as
  \begin{align}\nn
     \widetilde{N}_T \coloneqq \sum_{\stackrel{i\geq 1}{\xi_{i}\geq T}}\delta_{i\mu /\eta_T}.
 \end{align}
If the parameter of CTMC satisfies  $\beta<\frac{1}{\kappa}$ then 
 \begin{align*}
      \widetilde{N}_{T} \xrightarrow[T\to\infty]{\mathtt{d}} N^*.
 \end{align*}
\end{lemma}  

\begin{proof}
    The conclusion is through Theorem~\ref{leadbthm}, given that $\xi_i$ are iid, which results in  
    \begin{align*}
    \lim_{T\to\infty}\ex[\widetilde{N}_T(A)] = \lim_{T\to\infty} \frac{\eta_T}{\mu} \mathcal{L}(A) \pr[\xi_1\geq T] = \mathcal{L}(A) , 
    \end{align*}
    for any semi-closed interval $A\subset \rpl$ where $\mathcal{L}(A)$ denotes the Lebesgue measure of $A$. Checking \eqref{vacumprob} is also straightforward.    
\end{proof}

Now we are ready to prove Theorem~\ref{convpoissonthm} by relating $\widetilde{N}_T$ and $N'_T$.  
\begin{proof}[Proof of Theorem~\ref{convpoissonthm}]
  Given a semi-closed interval $I\coloneqq [t_1,t_2)\subset\rpl$, we will show that, \emph{with asymptotically high probability}, 
 \begin{align*}
     \widetilde{N}_T (I) = N'_T(I),
 \end{align*}
 that is,
 \begin{align*}
\lim_{T\to\infty}\pr[\widetilde{N}_T (I) \neq N'_T(I)] = 0,
 \end{align*}
 and then the proof will follow by Lemma~\ref{iffconvdlem}. We start by defining the following events 
 \begin{align}\nn
     A_{j,T}&\coloneqq \left\{\exists i: 1\leq i< \frac{t_j\eta_T}{\mu},\, \frac{\zeta_{i}}{\eta_T }\geq t_j,\, \xi_{i}\geq T \right\},\\
     B_{j,T}&\coloneqq \left\{\exists i: i\geq  \frac{t_j\eta_T}{\mu},\, \frac{\zeta_{i}}{\eta_T }< t_j,\, \xi_{i+1}\geq T \right\}, \nn
 \end{align}
 for $j\in\{1,2\}$. Note that,
 \begin{align*}
     A_{1,T}^\complement &\subseteq \left\{\forall i: 1\leq i< \frac{t_1\eta_T}{\mu},\, \xi_{i}\geq T\;\;\text{implies}\;\; \frac{\zeta_{i}}{\eta_T }< t_1  \right\}\\
     &\subseteq \left\{\forall i: 1\leq i< \frac{t_1\eta_T}{\mu},\, \xi_{i}\geq T\;\;\text{implies}\;\;  \frac{\zeta^T_{\upsilon(i-1)}}{\eta_T } <t_1 \right\}\\ 
     &\subseteq \{\widetilde{N}_T([0,t_1)) \leq N'_T([0,t_1))\}, 
 \end{align*}
 where with $A_{1,T}^\complement$, we denote the complement of the event $A_{1,T}$. The last implication follows from that, by definition, there is a bijection between the support sets of $\widetilde{N}_T$ and $N'_T$ given by 
 \begin{align*}
    \widetilde{N}_T (\{i\mu /\eta_T\})>0 \quad\text{if and only if}\quad N'_T (\{\zeta^T_{\upsilon(i-1)}/\eta_T\})>0
 \end{align*}
 for all $i\in\zpl\setminus\{0\}$. Similarly,
\begin{align*}
     A_{2,T}^\complement &\subseteq \left\{\forall i: 1\leq i< \frac{t_2\eta_T}{\mu},\, \xi_{i}\geq T\;\;\text{implies}\;\; \frac{\zeta^T_{\upsilon(i-1)}}{\eta_T }< \frac{\zeta_{i}}{\eta_T }< t_2  \right\}\\ 
     &\subseteq \{\widetilde{N}_T([0,t_2)) \leq N'_T([0,t_2))\}. 
 \end{align*}
Also,
\begin{align*}
     B_{1,T}^\complement &\subseteq \left\{\forall i : i\geq  \frac{t_1\eta_T}{\mu} ,\, \xi_{i+1}\geq T\;\;\text{implies}\;\; \frac{\zeta_{i}}{\eta_T }\geq t_1   \right\}\\
     &\subseteq \left\{\forall i : i\geq  \frac{t_1\eta_T}{\mu} ,\, \xi_{i+1}\geq T\;\;\text{implies}\;\; \frac{\zeta_{\nu(i)}}{\eta_T }\geq t_1 \right\}\\
     &\subseteq \{\widetilde{N}_T([0,t_1)) \geq N'_T([0,t_1))\}, 
 \end{align*}
and 
\begin{align*}
     B_{2,T}^\complement \subseteq \{\widetilde{N}_T([0,t_2)) \geq N'_T([0,t_2))\}. 
 \end{align*}
Therefore 
 \begin{align*}
     A_{1,T}^\complement\cap A_{2,T}^\complement \cap B_{1,T}^\complement\cap B_{2,T}^\complement \subseteq \{\widetilde{N}_T (I) = N'_T(I)\},  
 \end{align*}
 and
 \begin{align}\label{NTneqbdeqn}
     \pr[N'_T(I)\neq \widetilde{N}_T(I)] &\leq \pr[A_{1,T}] + \pr[B_{1,T}] + \pr[A_{2,T}] + \pr[B_{2,T}]. 
 \end{align}
We will show that each term on the right hand side above has limit 0 as $T\to\infty$. We first consider $\pr[A_{1,T}]$ and write,
\begin{align}
    \pr[A_{1,T}] &\leq \sum_{i=1}^{\frac{t_1\eta_T}{\mu}} \pr[\zeta_{i} \geq t_1 \eta_T,\, \xi_i\geq T]= \frac{\mu}{\eta_T}\sum_{i=1}^{\frac{t_1\eta_T}{\mu}} \pr[\zeta_{i} \geq t_1 \eta_T]\label{A1Tineq}
\end{align}
due to \rev{union} bound and using that $\xi_i$ are iid with \eqref{etaTdefeqn}. Now, choose some $\epsilon>0$. Note that there exists  $K_\epsilon\in (0,\infty)$ which satisfies 
\begin{equation*}
    \sum_{i=1}^{\infty} \pr\left[\frac{\zeta_{i}}{i} \geq\mu+\epsilon\right]\leq K_\epsilon
\end{equation*}
due to the Strong Law of Large Numbers (SLLN). 
Choose 
\begin{equation*}
    \rho_{T,\epsilon} \coloneqq \frac{t_1}{\mu(\mu+\epsilon)}\epsilon\eta_T.
\end{equation*}
From \eqref{A1Tineq} we get,
\begin{align}\nn
    \pr[A_{1,T}] &\leq \frac{\mu}{\eta_T} \sum_{i=1}^{\lfloor t_1\eta_T/\mu-\rho_{T,\epsilon}\rfloor} \pr[\zeta_i \geq t_1 \eta_T] + \frac{\mu}{\eta_T} \sum_{i=\lfloor t_1\eta_T/\mu-\rho_{T,\epsilon}\rfloor+1}^{ t_1\eta_T/\mu} \pr[\zeta_i \geq t_1 \eta_T]\\
    &\leq \frac{\mu}{\eta_T} \sum_{i=1}^{\lfloor t_1\eta_T/\mu-\rho_{T,\epsilon}\rfloor} \pr\left[\frac{\zeta_i}{i} \geq \frac{t_1 \eta_T}{t_1\eta_T/\mu-\rho_{T,\epsilon}}\right] +  \frac{\mu}{\eta_T}\rho_{T,\epsilon} \nn\\
    &\leq \frac{\mu}{\eta_T}\sum_{i=1}^{\infty} \pr\left[\frac{\zeta_i}{i} \geq \frac{t_1 \eta_T}{t_1\eta_T/\mu-t_1\epsilon\eta_T/\mu\cdot(\mu+\epsilon)}\right] +t_1\frac{\epsilon}{\mu+\epsilon}\nn\\
    &\leq \frac{\mu}{\eta_T}\sum_{i=1}^{\infty} \pr\left[\frac{\zeta_i}{i} \geq \mu+\epsilon\right] +t_1\frac{\epsilon}{\mu+\epsilon}\leq \frac{\mu}{\eta_T} K_\epsilon + t_1\epsilon.\label{proba1tbd} 
\end{align}
Since $\lim_{T\to\infty} \eta_T = \infty$, by choosing $T$ large enough, the right hand side of the above inequality can be made as close to $t_1\epsilon$ as desired. Since \eqref{proba1tbd} holds for all $\epsilon>0$, we conclude that
\begin{align*}
    \lim_{T\to\infty} \pr[A_{1,T}] = 0. 
\end{align*}
For $\pr[B_{1,T}]$, using Markov bound again and that $\xi_i$ are iid, we write 
\begin{align}
    \pr[B_{1,T}] &\leq \frac{\mu}{\eta_T} \sum_{i=\lceil t_1\eta_T/\mu\rceil}^{\lfloor t_1\eta_T/\mu+\varrho_{T,\epsilon}\rfloor} \pr[\zeta_i < t_1 \eta_T] + \frac{\mu}{\eta_T} \sum_{i=\lfloor t_1\eta_T/\mu+\varrho_{T,\epsilon}\rfloor+1}^\infty \pr[\zeta_i < t_1 \eta_T]\label{b1tbdeqn} 
\end{align}
where
\begin{align*}
\varrho_{T,\epsilon}\coloneqq \frac{t_1}{\mu(\mu-\epsilon)}\epsilon\eta_T.
\end{align*}
We can find a $K'_\epsilon\in \rpl$, due to SLLN, such that 
\begin{equation}\nn
    \sum_{i=1}^{\infty} \pr\left[\frac{\zeta_{i}}{i} <\mu-\epsilon\right]\leq K'_\epsilon.
\end{equation}
Continuing from \eqref{b1tbdeqn}, using the inequality above,
\begin{align*}
    \pr[B_{1,T}] 
    &\leq \frac{\mu}{\eta_T}\varrho_{T,\epsilon} +  \frac{\mu}{\eta_T} \sum_{i=\lfloor t_1\eta_T/\mu+\varrho_{T,\epsilon}\rfloor+1}^\infty \pr\left[\frac{\zeta_i}{i} <\frac{t_1 \eta_T}{t_1\eta_T/\mu+\varrho_{T,\epsilon}}\right]  \\
    &\leq t_1\frac{\epsilon}{\mu-\epsilon} + \frac{\mu}{\eta_T}\sum_{i=1}^{\infty} \pr\left[\frac{\zeta_i}{i} <\mu-\epsilon\right] \leq t_1\epsilon + \frac{\mu}{\eta_T} K'_\epsilon.
\end{align*}
From which, we conclude that $\lim_{T\to\infty} \pr[B_{1,T}]=0$. Similarly,
 \begin{align*}
    \lim_{T\to\infty} \pr[A_{2,T}] = \lim_{T\to\infty} \pr[B_{2,T}] =  0. 
\end{align*}
Therefore, the proof concludes through \eqref{NTneqbdeqn} and Lemma~\ref{iffconvdlem}.  
\end{proof}
Note that we have not exactly characterized the asymptotics of the quantity $\eta_T$ as $T\to\infty$. 
Nevertheless, we can state the following asymptotic lower bound.

\begin{lemma}
    If $\beta<\frac{1}{\kappa}$ then,  \begin{equation*}
    \liminf_{T\to\infty} (\eta_T)^{1/T} \geq  \frac{1}{\kappa\beta}. 
    \end{equation*}    
\end{lemma}
\begin{proof}
We use the electrical network representation (see \cite{peresbook}, Chapter-9 for definitions) of the embedded discrete-time Markov chain, $\hat{\bs X}_i$, to bound the \emph{escape probability},
\begin{align*}
    \pr[\xi_1\geq T] = \pr[\zeta^T_{1}<\zeta_1]. 
\end{align*}

Note that for any pair of clumps $\bs Y, \bs Z$, such that $|\card(\bs Z) - \card(\bs Y)| \allowbreak = 1$, the transition probability of $\hat{\bs X}_i$ satisfies, 
 \begin{align*}
     \pr(\hat{\bs X}_{i+1}=\bs Z| \hat{\bs X}_{i}= \bs Y) =\begin{cases}
         \frac{p_\uparrow}{p_\downarrow M^-(\bs Y) + p_\uparrow M^+(\bs Y) } &\text{if $\bs Z\setminus \bs Y\in \mc M^+(\bs Y)$,}\\
      \frac{p_\downarrow}{p_\downarrow M^-(\bs Y) + p_\uparrow M^+(\bs Y)} &\text{if $\bs Y\setminus \bs Z\in \mc M^-(\bs Y)$,}\\
      0 &\text{otherwise.}      
     \end{cases} 
 \end{align*}
for $i\in\zpl$, where
\begin{align}
    \mc M^+(\bs Y)&\coloneqq\{Y\in \mc M(\bs Y)\setminus \bs Y: \text{$\bs Y\cup Y$ is a clump}\}\nn\\
    \mc M^-(\bs Y)&\coloneqq\{Y\in \bs Y: \text{$\bs Y\setminus Y$ is a clump}\}, \nn
\end{align}
and 
\begin{align*}
    M^+(\bs Y)&\coloneqq\card(\mc M^+(\bs Y))\\
    M^-(\bs Y)&\coloneqq \card(\mc M^-(\bs Y)).
\end{align*}
Therefore, the detailed balance equation
\begin{align*}
    \pr(\hat{\bs X}_{i+1}=\bs Z| \hat{\bs X}_{i}= \bs Y) \hat \Pi (\bs Y) &= \pr(\hat{\bs X}_{i+1}=\bs Y| \hat{\bs X}_{i}= \bs Z) \hat \Pi (\bs Z)\\
    &= K^{-1} \beta^{\max\{\card(\bs Y), \card(\bs Z)\}}
\end{align*}
holds for the stationary distribution $\hat \Pi$ defined as 
\begin{align}\nn
    \hat \Pi (\bs Y) \coloneqq  K^{-1}\cdot \beta^{\card(\bs Y)} \left(\beta M^+(\bs Y) + M^-(\bs Y) \right),
\end{align}
where 
\begin{align*}
    K&\coloneqq \sum_{\text{$\bs Y$ is a clump}}\beta^{\card(\bs Y)} \left(\beta M^+(\bs Y) + M^-(\bs Y) \right) \\
    &\leq (\beta+1) \sum_{\text{$\bs Y$ is a clump}}\beta^{\card(\bs Y)} 9\card(\bs Y)< \infty,
\end{align*}
due to Lemma~\ref{pigoodlem}. Therefore, the electrical network representation of $\hat{\bs X}_i$ between the clumps $\bs Y$ and $\bs Z$ has conductance 
\begin{align*}
    \mc C(\bs Y,\bs Z) = \begin{cases}
    \frac{\beta^{\max\{\card(\bs Y), \card(\bs Z)\}}}{K} \,&\text{if $\bs Y$ is indentable to $\bs Z$ or vice versa,}\\
    0 &\text{otherwise.}
    \end{cases}
\end{align*}

The network diagram of $\hat{\bs X}_i$ is shown in Figure~\ref{elecnetwfig}. 
\begin{figure}[ht]
\centering
\scalebox{.5}{\begin{tikzpicture}[scale=1, line join=round, line cap=round]


\foreach \x in {-1,0,1}
\foreach \y in {-1,0,1}
{
        \draw[very thick, fill=gray!10,dashed] (\x,\y) rectangle ++(1,1);
}
 \draw[very thick, fill=gray!40] (0,0) rectangle ++(1,1);
\fill[black] (0.5,0.5) circle (1pt);
\node[anchor=north] (Label) at (0.5, 0.55) {$(0,0)$};

\tikzmath{\xshift = 9;}
\foreach \z in {7.5, 2.5, -2.5, -7.5}
{
\foreach \x in {-1,0,1}
\foreach \y in {-1,0,1}
{
        \draw[very thick, fill=gray!10,dashed] (\x+\xshift,\y+\z) rectangle ++(1,1);
}
 \draw[very thick, fill=gray!40] (\xshift,\z) rectangle ++(1,1);
 \fill[black] (\xshift + 0.5,\z+0.5) circle (1pt);
\node[anchor=north] (Label) at (\xshift+0.5,\z+ 0.55) {$(0,0)$};

\draw[thick](2.1,0.5) -- node[pos=0.55,sloped, above] {$\mc C=\beta^2/K$} (1.9+\xshift-3,0.5+\z);

\fill[black] (\xshift+4.5,0.5+\z) circle (2pt);
 \fill[black] (\xshift+5,0.5+\z) circle (2pt);
 \fill[black] (\xshift+5.5,0.5+\z) circle (2pt);

\foreach \t in {-2.5,-1.5,-0.5,0.5,1.5}
{
\draw[thick](2.1+\xshift,0.5+\z) -- (3.6+\xshift,0.5+\z+\t/2);
}
\draw[thick](2.1+\xshift,0.5+\z) -- node[pos=0.65,sloped, above] {$\mc C=\beta^3/K$} (3.6+\xshift,0.5+\z+2.5/2);

}

\draw[very thick, fill=gray!40] (\xshift+1,7.5) rectangle ++(1,1);
\draw[very thick, fill=gray!40] (\xshift,2.5+1) rectangle ++(1,1);
\draw[very thick, fill=gray!40] (\xshift-1,-2.5) rectangle ++(1,1);
\draw[very thick, fill=gray!40] (\xshift,-7.5-1) rectangle ++(1,1);

\node[anchor=north] (Label) at (0.5, -9) {\Large layer-1};
\node[anchor=north] (Label) at (0.5+\xshift, -9) {\Large layer-2};
 
\end{tikzpicture}}
\caption{Electrical network representation of $\hat{\bs X}_i$.}  
\label{elecnetwfig}
\end{figure}
From Proposition 9.5 of \cite{peresbook},
\begin{align}\nn
    \pr[\zeta^T_1<\zeta_1] = \frac{\mc C_{\text{eff}}(T)}{\hat \Pi \left([-1/2,1/2]^2 \right)},
\end{align}
where $\mc C_{\text{eff}}(T)$ denotes the effective conductance between the single square at the origin, i.e., layer-1, and the layer-$\lceil T \rceil$.

Using Corollary~9.13 of \cite{peresbook}, $\mc C_{\text{eff}}(T)$ is upper bounded by the modified network of $\hat{\bs X}_i$ where all the clumps with the same cardinality, i.e., those on the same layer, are glued together. The effective conductance of the modified network is calculated below using the parallel and series laws,      
\begin{align*}
    \mc C_{\text{eff}}(T)&\leq K^{-1}\left(\sum_{n=2}^{\lceil T\rceil }\frac{\beta^{-n}}{\sum_{\card(\bs Y)=n} M^{-}(\bs Y)}\right)^{-1} \leq 
    K^{-1}\left(\sum_{n=2}^{\lceil T\rceil }\frac{\beta^{-n}}{n\cdot n \cdot a_n}\right)^{-1} , 
\end{align*}
where the second inequality is due to $M^{-}(\bs Y)\leq \card(Y)$ and \eqref{sumindyclump}. 
By the definition of $\eta_T$, and the inequality above,
\begin{align*}
   \eta_T \geq K\mu \sum_{n=2}^{\lceil T\rceil }\frac{\beta^{-n}}{n^2 a_n}\geq \frac{C}{\beta^TT^2a_T} 
\end{align*}
for some constant $C>0$.
Therefore, the statement of the Lemma follows by \eqref{anlimeqn}.
\end{proof}

We conjecture below that, indeed, the limit, 
$ \lim_{T\to\infty}\eta_T^{1/T}$ 
exists. Proving it through a non-trivial upper bound will likely require a more delicate study of the state transition space of $\hat{\bs X}_i$, and is left for future work.

\begin{conjecture}
    If $\beta<\frac{1}{\kappa}$ then  \begin{equation*}
    \lim_{T\to\infty} \eta_T^{1/T} = \frac{1}{\kappa\beta}. 
    \end{equation*}
\end{conjecture}

\bibliography{references}  
\end{document}